\documentclass[11pt]{amsart}
\usepackage[lmargin=1.25in, rmargin=1.25in, tmargin=1.25in,bmargin=1.25in]{geometry}
\usepackage{amsthm}
\usepackage{amssymb}
\usepackage{MnSymbol}
\usepackage{amsxtra}     
\usepackage{epsfig}
\usepackage{verbatim}
\usepackage{url}
\usepackage[all]{xypic}

\theoremstyle{plain}
\newtheorem{thm}{Theorem}
\newtheorem{prop}[thm]{Proposition}
\newtheorem{cor}[thm]{Corollary}
\newtheorem{lem}[thm]{Lemma}

\theoremstyle{definition}

\theoremstyle{remark}
\newtheorem{rmk}[thm]{Remark}

\include{header}

\newcommand{\adeles}{\mathbb{A}}

\newcommand{\Ind}{\mathrm{Ind}}

\newcommand{\hn}{\mathcal{H}_n}
\newcommand{\A}{\univav}
\newcommand{\OK}{\mathcal{O}_{\cmfield}}
\newcommand{\Oe}{\mathcal{O}_E}
\newcommand{\Hom}{\mbox{Hom}}
\newcommand{\dual}{^\vee}
\newcommand{\isomto}{\overset{\sim}{\rightarrow}}

\newcommand{\ci}{C^{\infty}}

\newcommand{\uo}{\underline{\omega}}

\newcommand{\padic}{p\mathrm{-adic}}

\newcommand{\cmfield}{K}

\newcommand{\uA}{\underline{A}}

\newcommand{\gln}{\mathrm{GL}_n}
\newcommand{\gl}{\mathrm{GL}}

\newcommand{\e}{{\mathbf e}}

\newcommand{\IR}{\mathbb{R}}
\newcommand{\ZZ}{\mathbb{Z}}
\newcommand{\IC}{\mathbb{C}}
\newcommand{\IQ}{\mathbb{Q}}
\newcommand{\univav}{A_{\mathrm{univ}}}

\newcommand{\End}{\mbox{End}}

\newcommand{\Cp}{\IC_p}
\newcommand{\OCp}{\mathcal{O}_{\Cp}}

\newcommand{\Gal}{\mathrm{Gal}}

\newcommand{\hern}{\mathrm{Her}_n}
\newcommand{\Sgn}{\mathrm{Sign}}

\newcommand{\Sh}{\mathrm{Sh}}
\newcommand{\V}{\mathcal{V}}
\newcommand{\Vnn}{\V_{n,n}}
\newcommand{\ellcan}{\ell_{\mathrm{can}}}

\renewcommand{\paragraph}[1]{\noindent\textbf{#1}. }

\def\tr{{\rm tr}\,}

\begin{document}

\bibliographystyle{amsalpha}       
\title[An Eisenstein measure for vector-weight automorphic forms]{A $p$-adic Eisenstein Measure for vector-weight automorphic forms}
\author{Ellen Eischen}
\thanks{The author is partially supported by National Science Foundation Grant DMS-1249384.}
\address{Ellen Eischen\\
Department of Mathematics\\
The University of North Carolina at Chapel Hill\\
CB \#3250\\
Chapel Hill, NC 27599-3250\\
USA}
\email{eeischen@email.unc.edu}
\maketitle

\begin{abstract}
We construct a $p$-adic Eisenstein measure with values in the space of vector-weight $p$-adic automorphic forms on certain unitary groups. This measure allows us to $p$-adically interpolate special values of certain vector-weight $\ci$-automorphic forms, including Eisenstein series, as their weights vary.  This completes a key step toward the construction of certain $p$-adic $L$-functions.

We also explain how to extend our methods to the case of Siegel modular forms and how to recover Nicholas Katz's $p$-adic families of Eisenstein series for Hilbert modular forms.
\end{abstract}

\setcounter{tocdepth}{2}
\tableofcontents
\setcounter{secnumdepth}{3}

\section{Introduction}
The significance of $p$-adic families of Eisenstein series as a tool in number theory, especially for the construction of $p$-adic $L$-functions, is well-established.  For example, $p$-adic families of Eisenstein series play a key role in constructions of $p$-adic $L$-functions completed by Pierre Deligne, Nicholas Katz, Kenneth Ribet, and Jean-Pierre Serre \cite{serre, DR, kaCM}.  In a completely different direction, $p$-adic families of Eisenstein series also play a role in homotopy theory \cite{hopkins94, hopkinsicm, AHR}.

Each of the constructions mentioned above concerns only automorphic forms of scalar weight.  Automorphic forms on groups of rank $1$ (for example, modular forms and Hilbert modular forms, which are the forms with which Katz, Deligne, Ribet, and Serre worked) can only have scalar weights.  Automorphic forms on groups of higher rank, however, need not have scalar weights.  

By a vector-weight automorphic form, we mean an automorphic form whose weight is an irreducible representation with highest weight $\lambda_n\geq \cdots\geq \lambda_1$ is not required to have $\lambda_i = \lambda_{i+1}$ for all $i$, i.e. an automorphic form whose weight is not required to be a one-dimensional representation.  In order to complete a construction of $p$-adic $L$-functions for automorphic forms on unitary groups in full generality as in \cite{EEHLS}, one needs a $p$-adic Eisenstein measure that takes values in the space of $p$-adic vector-weight automorphic forms.  (By an {\it Eisenstein measure}, we mean a $p$-adic measure valued in a space of $p$-adic automorphic forms and whose values at locally constant functions are Eisenstein series.)  
 
The main result of this paper is the construction in Section \ref{emeasuresection} of a $p$-adic measure that takes values in the space of automorphic forms on unitary groups of signature $(n,n)$.  In particular, Theorem \ref{emeasurethm} gives a $p$-adic Eisenstein measure with values in the space of vector-weight automorphic forms.    As explained in Theorem \ref{AVmeasure}, this measure together with the results of Section \ref{diffopssection} allows us to $p$-adically interpolate the values of certain vector-weight $\ci$ (not necessarily holomorphic) automorphic forms, including Eisenstein series, as the (highest) weights of these automorphic forms vary.  Note that this is the first ever construction of a $p$-adic Eisenstein measure taking values in the space of {\it vector}-weight automorphic forms on unitary groups.

We follow the approach of \cite[Chapters 4 and 5]{kaCM} more closely than we did in \cite{AppToSHL}.  (In \cite{AppToSHL}, we constructed a $p$-adic Eisenstein measure for scalar-weight automorphic forms on unitary groups of signature $(n,n)$.)  As a consequence, in the final section of this paper, we easily recover Katz's Eisenstein measure from \cite[Chapters 4 and 5]{kaCM} as a special case of our results.

We also explain in Section \ref{symplectickatz} how to generalize the results of Section \ref{emeasuresection} to the case of Siegel modular forms, i.e. automorphic forms on symplectic groups.  In that setting, in the case where $n=1$, we are in exactly the situation of \cite{kaCM}, in which Katz constructs a $p$-adic Eisenstein measure for Hilbert modular forms.  As demonstrated in Section \ref{kacmnequals1}, the setup in the earlier sections of the paper makes the connection between our Eisenstein measure and the Eisenstein measure in \cite[Definition (4.2.5) and Equation (5.5.7)]{kaCM} almost transparent.

\subsection{Applications and context}
The main anticipated application of this paper is to the construction of $p$-adic $L$-functions for unitary groups, most immediately to \cite{EEHLS}.  While we do not construct the $p$-adic $L$-functions in this paper (since that is the topic of the joint paper \cite{EEHLS}), the results of this paper play a key role in the construction of those $p$-adic $L$-functions.  In particular, the $L$-functions in that paper are obtained through the ``doubling method'' (an approach described in \cite[Part A]{GPSR} and \cite[Section 2]{co}), which expresses values of $L$-functions in terms of values of Eisenstein series and values of cusp forms.  The $p$-adic Eisenstein measure in \cite[Section 4]{AppToSHL} suffices in the case of scalar weights, but if one does not restrict to scalar weights, one needs the results of the present paper.

The behavior of certain $L$-functions (for example, for unitary groups) is strongly tied to the behavior of certain Eisenstein series.  For instance, as explained in \cite[Introduction]{shar}, Shimura uses the algebraicity (up to a well-determined period) of values of Eisenstein series at CM points to prove the algebraicity (up to a well-determined period) of certain values of corresponding $L$-functions (normalized by a  period).  Analogously, as explained in \cite[Introduction]{kaCM}, Katz uses the $p$-adic interpolation of values of certain Eisenstein series (normalized by a period) at CM points to $p$-adically interpolate certain values of $L$-functions (normalized by a period).  Similarly, the $p$-adic families of Eisenstein series in the present paper play a key role in determining the behavior of the $L$-functions in \cite{EEHLS}.

\subsection{Overview and structure of the paper}

In Section \ref{conventionsbkdsection}, we introduce the conventions with which we will work, as well as standard background results necessary for this paper.  The conventions and background are similar to those in \cite[Section 2]{AppToSHL} and \cite{EDiffOps}.  The background is quite technical; we have summarized just what is needed for this paper.  The reader can find substantial reference materials on the background; for the reader seeking details on the background material, we recommend \cite{sh, shar} for the theory of $\ci$-automorphic forms and Eisenstein series on unitary groups, \cite{lanalgan, la} for the algebraic geometric background and a discussion of algebraically defined $q$-expansions, and \cite{hida, hi05} for the theory of $p$-adic automorphic forms.

In Section \ref{eseriesintrosection}, which relies in part on the results of \cite[Section 2]{AppToSHL}, we define certain scalar-weight Eisenstein series and automorphic forms on unitary groups of signature $(n,n)$.  Note that this set of automorphic forms includes the Eisenstein series defined in \cite[Section 2]{AppToSHL} but also includes other automorphic forms as well.  We need this larger space of automorphic forms in order to construct a $p$-adic measure with values in the space of vector-weight automorphic forms in Section \ref{emeasuresection}, whereas in \cite{AppToSHL}, we only were concerned with $p$-adic families of scalar-weight automorphic forms.  Like in \cite{AppToSHL}, we work adelically.  The formulation of the main result of this section (Theorem \ref{proplc}) is closer to that of \cite[Theorem (3.2.3)]{kaCM}, though, so that the reader can see parallels with the analogous construction in \cite[Section 3]{kaCM} (which is useful in Section \ref{kacmnequals1} when we compare our Eisenstein measure to the measure obtained in \cite[Definition (4.2.5) and Equation (5.5.7)]{kaCM}).

Section \ref{diffopssection} discusses differential operators that are necessary for comparing the values of certain $\ci$-automorphic forms and certain $\padic$-automorphic forms.  These differential operators are closely related to the differential operators discussed in \cite[Sections 8 and 9]{EDiffOps}.  Note that because we work with vector-weight automorphic forms, and not just scalar-weight automorphic forms, in this paper, we need more differential operators than we did in \cite{AppToSHL}, which handled only the case of scalar-weight automorphic forms.

Section \ref{emeasuresection} contains the main results of the paper, namely the construction of a $p$-adic Eisenstein measure and the $p$-adic interpolation of values of certain automorphic forms.  This is the heart of the paper.  The format of Section \ref{emeasuresection} closely parallels the construction of a $p$-adic Eisenstein measure in \cite[Sections 3.4 and 4.2]{kaCM}.  We also explain in Remark \ref{rmkapptoshl} precisely how the Eisenstein measure of \cite[Section 4]{AppToSHL} and the Eisenstein measure given in Theorem \ref{emeasurethm} are related.  Note that for $n\geq 2$, the measure in Theorem \ref{emeasurethm} is on a larger group than the the measure in \cite[Section 4]{AppToSHL}.  In order to construct a measure with values in the space of {\it vector}-weight automorphic forms without fixing a partition of $n$, this larger group is necessary.  (The approach in \cite{AppToSHL} relied on a choice of a partition of $n$, but it turns out that with this larger group, we do not need to fix a partition of $n$ and can consider a larger class of automorphic forms all at once.)  We also note that the construction of the measures in \cite[Section 4]{emeasurenondefinite} uses this measure as a starting point.

In Section \ref{symplectickatz}, we comment on how to extend the results of this paper to the case of Siegel modular forms, i.e. automorphic forms on symplectic groups.  The fact that our presentation in Section \ref{emeasuresection} closely follows the approach in \cite[Sections 3.4 and 4.2]{kaCM} also allows us to recover the Eisenstein measure of \cite[Definition (4.2.5) and Equation (5.5.7)]{kaCM} with ease in Section \ref{kacmnequals1}.

\subsection{Acknowledgements}

I would like to thank Chris Skinner for helpful conversations while working on this project.  I would also like to thank Kai-Wen Lan for clarifying my understanding of the $q$-expansion principle for automorphic forms on unitary groups of signature $(n,n)$ and symplectic groups.

\section{Conventions and Background}\label{conventionsbkdsection}
In Section \ref{conventionssection}, we introduce the conventions that we will use throughout the paper.  In Section \ref{backgroundsection}, we briefly summarize the necessary background on automorphic forms on unitary groups.  (A more detailed discussion of automorphic forms is presented in various references, including \cite{sh, shar, la, hida, EDiffOps}; details in the analogous case of Hilbert modular forms are covered in \cite[Section 1]{kaCM}.)

\subsection{Conventions}\label{conventionssection}
Once and for all, fix a CM field $K$ with maximal totally real subfield $E$.  Fix a prime $p$ that is unramified in $K$ and such that each prime of $E$ dividing $p$ splits completely in $K$.  Fix embeddings
\begin{align*}
\iota_\infty: &\bar{\IQ}\hookrightarrow\IC\\
\iota_p: &\bar{\IQ}\hookrightarrow\IC_p,
\end{align*}
and fix an isomorphism
\begin{align*}
\iota: \bar{\IC}_p\isomto\IC
\end{align*}
satisfying $\iota\circ\iota_p = \iota_{\infty}$.  From here on, we identify $\bar{\IQ}$ with $\iota_p(\bar{\IQ})$ and $\iota_{\infty}(\bar{\IQ})$.  Let $\OCp$ denote the ring of integers in $\IC_p$.

Fix a CM type $\Sigma$ for $K/\IQ$.  For each element $\sigma\in \Hom(E, \bar{\IQ})$, we also write $\sigma$ to denote the unique element of $\Sigma$ prolonging $\sigma:E\hookrightarrow\bar{\IQ}$ (when no confusion can arise).  For each element $x\in K$, denote by $\bar{x}$ the image of $x$ under the unique non-trivial element $\epsilon\in \Gal(K/E)$, and let $\bar{\sigma} = \sigma\circ\epsilon$.

Given an element $a$ of $E$, we identify it with an element of $E\otimes\IR$ via the embedding
\begin{align}\label{embconv}
E&\hookrightarrow E\otimes\IR\\
a&\mapsto (\sigma(a))_{\sigma\in\Sigma}.
\end{align}
We identify $a\in K$ with an element of $K\otimes \IC\isomto (E\otimes\IC)\times (E\otimes \IC)$ via the embedding
\begin{align*}
K&\hookrightarrow K\otimes\IC\\
a&\mapsto \left((\sigma(a))_{\sigma\in\Sigma}, (\bar{\sigma}(a))_{\sigma\in\Sigma}\right).
\end{align*}

Let $d = (d_v)_{v\in\Sigma}\in\ZZ^\Sigma$, and let $a = (a_v)_{v\in\Sigma}$ be an element of $\IC^\Sigma$ or $\IC_p^\Sigma$.  We denote by $a^d$ the element of $\IC$ or $\IC_p$ defined by
\begin{align*}
a^d := \prod_{v\in \Sigma}a_v^{d_v}.
\end{align*}
If $e = \left(e_v\right)_{v\in\Sigma}\in\ZZ^\Sigma$, we denote by $d+e$ the tuple defined by
\begin{align*}
d+e = \left(d_v+e_v\right)_{v\in\Sigma}\in\ZZ^\Sigma.
\end{align*}
If $k\in\ZZ$, we denote by $k+d$ or $d+k$ the element
\begin{align*}
k+d = d+k = \left(d_v+k\right)_{v\in\Sigma}\in\ZZ^\Sigma.
\end{align*}

For any ring $R$, we denote the ring of $n\times n$ matrices with coefficients in $R$ by $M_{n\times n}(R)$ or $M_{n\times n}(R)$.  We denote by $1_n$ the multiplicative identity in $M_{n\times n}(R)$.  Also, for any subring $R$ of $K\otimes_E E_v$, with $v$ a place of $E$, let $\hern(R)$ denote the space of Hermitian $n\times n$-matrices with entries in $R$.  Given $x\in\hern(E)$,
\begin{align*}
x>0
\end{align*}
if $\sigma(x)$ is positive definite for every $\sigma\in\Sigma$.

\subsubsection{Adelic norms}
Let $\left|\cdot\right|_E$ denote the adelic norm on $E^\times\backslash\adeles_E^\times$ such that for all $a\in \adeles_E^\times$, 
\begin{align*}
\left|a\right|_E = \prod_v\left|a\right|_v,
\end{align*}
where the righthand product is over all places of $E$ and where the absolute values are normalized so that
\begin{align*}
\left|v\right|_v & = q_v^{-1},\\
q_v & = \mbox{ the cardinality of } {\Oe}_v/v{\Oe}_v,
\end{align*}
for all non-archimedean primes $v$ of the totally real field $E$.  Consequently for all $a\in E$,
\begin{align*}
\prod_{v\ndivides\infty}|a|_v^{-1} = \prod_{v\in\Sigma}\sigma_v(a)\Sgn(\sigma_v(a)),
\end{align*}
where the product is over all archimedean places $v$ of the totally real field $E$.  We denote by $\left|\cdot\right|_K$ the adelic norm on $K^\times\backslash\adeles_K^\times$ such that for all $a\in \adeles_K^\times$,
\begin{align*}
\left|a\right|_K = \left|a\bar{a}\right|_E.
\end{align*}
For $a\in K$ and $v$ a place of $E$, we let
\begin{align*}
\left|a\right|_v = \left|a\bar{a}\right|_v^{\frac{1}{2}}.
\end{align*}
Given an element $a\in K$, we associate $a$ with an element of $K\otimes \IR$, via the embedding
\begin{align*}
a\mapsto (\sigma(a))_{\sigma\in \Sigma}.
\end{align*}

For any field extension $L/M$, we write $\mathbf{N}_{L/M}$ to denote the norm from $L$ to $M$.  Given a $\mathcal{O}_M$-algebra $R$, the norm map $\mathbf{N}_{L/M}$ on $L$ provides a group homomorphism
\begin{align*}
\left(\mathcal{O}_L\otimes R\right)^\times&\rightarrow R^\times
\end{align*}
in which $a\otimes r\mapsto \mathbf{N}_{L/M}(a)r$.  When the fields are clear, we shall just write $\mathbf{N}$.

\subsubsection{Exponential characters}

For each archimedean place $v\in\Sigma$, denote by $\e_v$ the character of $E_v$ (i.e $\IR$) defined by
\begin{align*}
\e_v(x_v) = e^{(2\pi i x_v)}
\end{align*}
for all $x_v$ in $E_v$.  Denote by $\e_{\infty}$ the character of $E\otimes \IR$ defined by
\begin{align*}
\e_{\infty}((x_v)_{v\in \Sigma}) = \prod_{v\divides\infty}\e_v(x_v).
\end{align*}
Following our convention from \eqref{embconv}, we put
\begin{align*}
\e_{\infty}(a) = \e_{\infty}((\sigma(a))_{\sigma\in\Sigma}) = e^{2\pi i\tr_{E/\IQ}(a)}
\end{align*}
for all $a\in E$.
For each finite place $v$ of $E$ dividing a prime $q$ of $\ZZ$, denote by $\e_v$ the character of $E_v$ defined for each $x_v\in E_v$ by
\begin{align*}
\e_v(x_v) = e^{-2\pi i y}
\end{align*}
where $y\in \IQ$ is the fractional part of $\tr_{E_v/\IQ_q}(x_v)\in \IQ_p$; i.e. writing $\tr_{E_v/\IQ_q}(x_v) = \sum_{i=k}^\infty a_i p^i$ for some integer $k\leq 0$ and $a_i\in \left\{0, \ldots, p-1\right\}$, then $y = \sum_{i=k}^0 a_i p^i$.  We denote by $\e_{\adeles_E}$ the character of $\adeles_E$ defined by
\begin{align*}
\e_{\adeles_E}(x) = \prod_v\e_v\left(x_v\right)
\end{align*}
for all $x = \left(x_v\right)\in \adeles_E$.  
\begin{rmk}
Note that for a $a\in E$, we identify $a$ with the element $(\sigma_v(a))_v\in\adeles_E$, where $\sigma_v: E\hookrightarrow E_v$ is the embedding corresponding to $v$.  Following this convention, we put
\begin{align}\label{ewarning}
\e_{\adeles_E}(a) = \prod_v\e_v(\sigma_v(a)).
\end{align}
for all $a\in E$.
\end{rmk}

\subsubsection{Spaces of functions}
Given topological spaces $X$ and $Y$, we let
\begin{align*}
\mathcal{C}(X, Y)
\end{align*}
denote the space of continuous functions from $X$ to $Y$.

\subsection{Background concerning automorphic forms on unitary groups}\label{backgroundsection}

\subsubsection{Unitary groups of signature $(n,n)$}
We now recall basic information about unitary groups and automorphic forms on unitary groups.  (A more detailed discussion of unitary groups and automorphic forms on unitary groups appears in \cite{sh, shar, la, HLS, EDiffOps}; the analogous background for the case of Hilbert modular forms is the main subject of \cite[Section 1]{kaCM}.)

The material in this section is similar to the material in \cite[Section 2.1]{AppToSHL}.  Although we discussed embeddings of non-definite unitary groups of various signatures into unitary groups of signature $(n,n)$ in \cite[Section 2.1]{AppToSHL}, we are primarily concerned only with unitary groups of signature $(n, n)$ and definite unitary groups in this paper; in the sequel, \cite{emeasurenondefinite}, we discuss pullbacks to various products of unitary groups occurring as subgroups.

Let $V$ be a vector space of dimension $n$ over the CM field $K$, and let $\langle, \rangle_V$ denote a positive definite hermitian pairing on $V$.  Let $-V$ denote the vector space $V$ with the negative definite hermitian pairing $-\langle, \rangle_V$.  Let
\begin{align*}
W& = 2V = V\oplus -V\\
\langle \left(v_1, v_2\right), \left(w_1, w_2\right)\rangle_W & = \langle v_1, w_1\rangle_V+\langle v_2, w_2\rangle_{-V}.
\end{align*}

The hermitian pairing $\langle, \rangle_W$ defines an involution $g\mapsto \tilde{g}$ on $\End_K(W)$ by
\begin{align*}
\langle g(w), w'\rangle_W = \langle w, \tilde{g}(w')\rangle_W
\end{align*}
(where $w$ and $w'$ denote elements of $W$).  Note that this involution extends to an involution on $\End_{K\otimes_E R}\left(W\otimes_E R\right)$ for any $E$-algebra $R$.
We denote by $U$ the algebraic group such that for any $E$-algebra $R$, the $R$-points of $U$ are given by
\begin{align*}
U(R) = U(R, W) = \left\{ g\in \gl_{K\otimes_E R}\left(W\otimes_E R\right)\middle|  g\tilde{g} = 1\right\}.
\end{align*}
Similarly, we define $U(R, V)$ to be the algebraic group associated to $\langle, \rangle_V$ and $U(R, -V)$ to be the algebraic group associated to $\langle, \rangle_{-V}$.  Note that $U(\IR)$ is of signature $(n,n)$.  Also, note that the canonical embedding
\begin{align*}
V\oplus V\hookrightarrow W
\end{align*}
induces an embedding
\begin{align*}
U(R, V)\times U(R, -V)\hookrightarrow U(R, W)
\end{align*}
for all $E$-algebras $R$.  When the $E$-algebra $R$ over which we are working is clear from context or does not matter, we shall write $U(W)$ for $U(R, W)$, $U(V)$ for $U(R, V)$, and $U(-V)$ for $U(R, -V)$.  We also sometimes write just $U$ to denote $U(W)$.

We also have groups 
\begin{align*}
GU(R) = GU(R, W) = \left\{ g\in \gl_{K\otimes_E R}\left(W\otimes_E R\right)\middle|  g\tilde{g}\in R^\times\right\}.
\end{align*}
We use the notation $\omega$ to denote the similitude character
\begin{align*}
\omega: GU(R)&\rightarrow R^\times\\
g&\mapsto g\tilde{g}.
\end{align*}
When the $E$-algebra $R$ over which we are working is clear from context or does not matter, we shall write $GU(W)$ for $GU(R, W)$.  We shall also use the notation
\begin{align*}
G(R) = GU(R, W)
\end{align*}
or write simply $G$ or $GU$ when the ring $R$ is clear from context or does not matter.  When $R = \adeles_E$ or $R=\IR$, we write
\begin{align*}
G_+:=GU_+
\end{align*}
to denote the subgroup of $G = GU$ consisting of elements such that the similitude factor at each archimedean place of $E$ is positive.

For the space $W = V\oplus -V$ defined above, $U(W)$ and $GU(W)$ have signature $(n,n)$.  So we will sometimes write $U(n,n)$ and $GU(n,n)$, respectively, to refer to these groups.

We write $W = V_d\oplus V^d,$ where $V_d$ and $V^d$ denote the maximal isotropic subspaces
\begin{align*}
V^d & = \left\{(v, v)|v\in V\right\}\\
V_d & = \left\{(v, -v)| v\in V\right\}.
\end{align*}
Let $P$ be the Siegel parabolic subgroup of $U(W)$ stabilizing $V^d$ in $V_d\oplus V^d$ under the action of $U(W)$ on the right.  Denote by $M$ the Levi subgroup of $P$ and by $N$ the unipotent radical of $P$.  Similarly, denote by $GP$ the Siegel parabolic subgroup of $GU(W)$ stabilizing $V^d$ in $V_d\oplus V^d$ under the action of $GU(W)$ on the right, and denote by $GM$ the Levi subgroup of $GP$ and by $N$ the unipotent radical of $GP$.  We also, similarly, denote by $GP_+$ the Siegel parabolic subgroup of $GU_+$ stabilizing $V^d$ in $V_d\oplus V^d$ under the action of $GU_+$ on the right, and denote by $GM_+$ the Levi subgroup of $GP_+$ and by $N$ the unipotent radical of $GP_+$.

A choice of a basis $e_1, \ldots, e_n$ for $V$ over $K$ gives an identification of $V$ with $V^d$ (via $e_i\mapsto \left(e_i, e_i\right)$) and with $V_d$ (via $e_i\mapsto \left(e_i, -e_i\right)$).  The choice of a basis for $V$ also identifies $\gl_K(V)$ with $\gl_n(K)$.  With respect to the ordered basis $\left(e_1, e_1\right)\ldots, \left(e_n, e_n\right), \left(e_1, -e_1\right)\ldots, \left(e_n, -e_n\right)$ for $W$, $M$ consists of the block diagonal matrices of the form
\begin{align*}
m(h) := \left({ }^t\overline{h}^{-1}, h\right)
\end{align*}
 with $h\in \gln(K\otimes R)$, and $GM$ consists of 
the block diagonal matrices of the form 
\begin{align*}
m(h, \lambda) := \left({ }^t\overline{h}^{-1}, \lambda h\right)
\end{align*}
 with $h\in \gln(K)$ and $\lambda\in E^\times$.  Thus, the choice of basis $e_1, \ldots, e_n$ for $V$ over $K$ fixes identifications
\begin{align*}
M&\isomto \gl_K(V)\\
GM&\isomto \gl_K(V)\times E^\times.
\end{align*}
Note that these isomorphisms extend to isomorphisms 
\begin{align}
M(R)&\isomto \gl_{K\otimes_ER}\left(V\otimes_ER\right)\label{leviglvunitary}\\
GM(R)&\isomto \gl_{K\otimes_ER}(V\otimes_ER)\times R^\times\label{leviglv}
\end{align}
for each $E$-algebra $R$.

We fix a Shimura datum $\left(G, X\left(W\right)\right)$ and a corresponding Shimura variety $\Sh(W) = \Sh(U(n,n))$, according to the conditions in \cite[Section 1.2]{HLS} and \cite[Section 2.2]{EDiffOps}.  Note that the symmetric domain $X(W)$ is holomorphically isomorphic to the tube domain consisting of $[E:\IQ]$ copies of
\begin{align*}
\hn=\left\{z\in M_{n\times n}(\IC)\middle| i({ }^t\bar{z}-z)>0\right\}.  
\end{align*}
When we need to emphasize over which ring $R$ we work, we sometimes write $\Sh(R)$.  Let $\mathcal{K}_{\infty}$ be the stabilizer in $G(\IR)$ of the point $i\cdot 1_n$.  So $\prod_{\sigma\in\Sigma}\mathcal{K}_{\infty}$ is the stabilizer in $\prod_{\sigma\in\Sigma}G(\IR)$ of the point
\begin{align}\label{gimeldefn}
\mathbf{i} = \left(i\cdot 1_n\right)_{\sigma\in\Sigma} \in \prod_{\sigma\in\Sigma}\hn.
\end{align}
Note that we can identify $G_+(\IR)/\mathcal{K}_{\infty}$ with $\hn$.  Given a compact open subgroup $\mathcal{K}$ of $G(\adeles_f)$, denote by $_{\mathcal{K}}\Sh(W)$ the Shimura variety whose complex points are given by
\begin{align*}
G(\IQ)\backslash X\times G(\adeles_f)/\mathcal{K}.
\end{align*}
This Shimura variety is a moduli space for abelian varieties together with a polarization, an endomorphism, and a level structure (dependent upon the choice of $\mathcal{K}$).  Note that $_{\mathcal{K}}\Sh(W)$ consists of copies of quotients of spaces isomorphic to $\hn$.  

When we are working with some other group $H$, we write $\Sh(H)$ instead of $\Sh(W)$.

\subsubsection{Automorphic Forms on unitary groups}

Automorphic forms on unitary groups are typically discussed from any of the following three perspectives (which are equivalent over $\IC$):
\begin{enumerate}
\item{Functions on a unitary group that satisfy an automorphy condition}
\item{$\ci$- (or holomorphic) functions on a hermitian symmetric space (analogue of the upper half plane) that satisfy an automorphy condition}
\item{Sections of a certain vector bundle over a moduli space (a Shimura variety) parametrizing abelian varieties together with a polarization, endomorphism, and level structure}
\end{enumerate}
Which perspective is most natural depends upon context.  In this paper, we shall need all three perspectives.  (In \cite[Section 2]{EDiffOps}, we provided a detailed discussion of automorphic forms and the relationships between different approaches to defining them.)

The relationship between the first two approaches to automorphic forms is reviewed in \cite[p. 9]{AppToSHL} and \cite[A8]{shar}.  The relationship between the second two approaches to automorphic forms is discussed in \cite[Section 2]{EDiffOps} and is similar to the analogous relationship for modular forms given in \cite[A1.1]{ka2}.

Note that an automorphic form $f$ on $U(n, n)$ has a weight, which is a representation $\rho$ of $\gln\times \gln$.  In the special case where this representation is of the form 
\begin{align*}
\rho(a, b) = \det(a)^{k+\nu}\det(b)^{-\nu},
\end{align*}
we shall say {\it $f$ is an automorphic form of weight $(k, \nu)$.}

As explained in \cite{la, lanalgan}, for the unitary groups of signature $(n,n)$, there is a higher-dimensional analogue of the Tate curve (which we call the ``Mumford object'' in \cite[Section 4.2]{EDiffOps} and \cite[Section 2.2.11]{AppToSHL}), and so in analogue with the case for modular forms evaluated at the Tate curve, one obtains an algebraic $q$-expansion by evaluating an automorphic form at the Mumford object.  Like in the case of modular forms, the coefficients of an algebraically defined $q$-expansion of a holomorphic automorphic form $f$ of over $\IC$ agree with the (analytically defined) Fourier coefficients of $f$ \cite{lanalgan}.  Also, like in the case of modular forms, there is a $q$-expansion principle for automorphic forms on unitary groups \cite[Prop 7.1.2.15]{la}; note that the $q$-expansion principle for automorphic forms over a Shimura variety requires the evaluation of an automorphic form at one cusp of each connected component.  As explained in \cite[Section 8.4]{hida}, to apply the $q$-expansion principle, it is enough to check the cusps parametrized by points of $GM_+(\adeles_E)$.  (The author is grateful to thank Kai-Wen Lan for explaining this to her.)  We shall say ``a cusp $m\in GM_+(\adeles_E)$'' to mean ``the cusp corresponding to the point $m$.''  Note that the $q$-expansion of an automorphic form at a cusp $m(h, \lambda)$ is a sum of the form
\begin{align*}
\sum_{\beta\in L_{m(h, \lambda)}}a(\beta)q^\beta,
\end{align*}
where $L_{m(h, \lambda)}$ is a lattice in $\hern(E)$ dependent upon the choice of the cusp $m(h, \lambda)$ and $a(\beta)\in\IC$ for all $\beta$ (or, more generally, if $f$ is a $V$-valued automorphic form for some $\IC$-vector space $V$, $a(\beta)\in V$ for all $\beta$).  We sometimes also write
\begin{align*}
\sum_{\beta\in \hern(E)}a(\beta)q^\beta,
\end{align*}
when we do not need to make the cusp explicit; in this case, we know that the coefficients $a(\beta)$ are zero outside of some lattice in $\hern(E)$ (namely, the lattice corresponding to the unspecified cusp).

Note that throughout the paper, all cusps $m$ and corresponding lattices $L_m\subseteq\hern(K)$ determined by $m$ are chosen so that the elements of $L_m$ have $p$-integral coefficients.\footnote{Even without this choice for $m$ and $L_m$, which we did not make {\it a priori} in \cite{AppToSHL}, we could force the Fourier coefficients at all the non-$p$-integral elements of $\hern(K)$ to be zero, simply by our choice of a Siegel section at $p$ later in this paper.  In fact, in \cite[Section 2.2]{AppToSHL}, our choice of Siegel sections at $p$ forced the Fourier coefficients at all the non-$p$-integral elements of $\hern(K)$ to be zero.}

\section{Eisenstein series on unitary groups}\label{eseriesintrosection}
In this section, we introduce certain Eisenstein series on unitary groups of signature $(n,n)$.  These Eisenstein series are related to the ones discussed in \cite[Section 2]{AppToSHL}, \cite[Section 18]{sh}, and \cite[Section (3.2)]{kaCM}.

For $k\in\ZZ$ and $\nu = \left(\nu(\sigma)\right)_{\sigma\in\Sigma}\in \ZZ^\Sigma$, we denote by $\mathbf{N}_{k, \nu}$ the function
\begin{align*}
\mathbf{N}_{k, \nu}: K^\times&\rightarrow K^\times\\
b&\mapsto\prod_{\sigma\in\Sigma}\sigma(b)^{k+2\nu(\sigma)}\left(\sigma(b)\bar{\sigma}(b)\right)^{-\left(\nu(\sigma)\right)}.
\end{align*}
Note that for all $b\in \Oe^\times$,
\begin{align*}
\mathbf{N}_{k, \nu}(b) = \mathbf{N}_{E/\IQ}^k(b).
\end{align*}

\begin{thm}\label{proplc}
Let $R$ be an $\OK$-algebra, let $\nu = \left(\nu(\sigma)\right)\in\ZZ^\Sigma$, and let $k\geq n$ be an integer.    Let
\begin{align*}
F: \left(\OK\otimes\ZZ_p\right)\times M_{n\times n}\left(\Oe\otimes\ZZ_p\right)\rightarrow R
\end{align*}
be a locally constant function supported on $\left(\OK\otimes\ZZ_p\right)^\times\times M_{n\times n}\left(\Oe\otimes\ZZ_p\right)$ that satisfies
\begin{align}\label{equnknualakacm}
F\left(ex, \mathbf{N}_{K/E}(e^{-1})y\right) = \mathbf{N}_{k, \nu}(e)F\left(x, y\right)
\end{align}
for all $e\in \OK^\times$, $x\in \OK\otimes\ZZ_p$, and $y\in M_{n\times n}\left(\Oe \otimes\ZZ_p\right)$.  There is an automorphic form $G_{k, \nu, F}$ (on $U(n,n)$) of weight $(k, \nu)$ defined over $R$ whose $q$-expansion at a cusp $m\in GM_+(\adeles_E)$ is of the form $\sum_{0<\beta\in L_m}c(\beta)q^\beta$ (where $L_{m}$ is the lattice in $\hern(K)$ determined by $m$), with $c(\beta)$ a finite $\ZZ$-linear combination of terms of the form
\begin{align*}
F\left(a, \mathbf{N}_{K/E}(a)^{-1}\beta\right)\mathbf{N}_{k, \nu}\left(a^{-1}\det\beta\right)\mathbf{N}_{E/\IQ}\left(\det\beta\right)^{-n}
\end{align*}
(where the linear combination is a sum over a finite set of $p$-adic units $a\in K$ dependent upon $\beta$ and the choice of cusp $m\in GM$).  When $R = \IC$, these are the Fourier coefficients at $s=\frac{k}{2}$ of the $\ci$-automorphic form $G_{k, \nu, F}\left(z, s\right)$ (which is holomorphic at $s=\frac{k}{2}$) that will be defined in Lemma \ref{lemprop8}.
\end{thm}
(Above, the elements of $\left(\Oe\otimes\ZZ_p\right)^\times$ in $M_{n\times n}\left(\Oe\otimes\ZZ_p\right)$ are viewed as homomorphisms, i.e. multiplication by an element of $\left(\Oe\otimes\ZZ_p\right)^\times$, so as diagonal matrices in $M_{n\times n}\left(\Oe\otimes\ZZ_p\right)$.  Also, note that when $\det\beta=0$, the coefficient of $q^\beta$ is $0$, so we can restrict the discussion to $F$ with support in $\left(\OK\otimes\ZZ_p\right)^\times\times \gln\left(\Oe\otimes\ZZ_p\right)$.)

\begin{proof}
By an argument similar to Katz's argument at the beginning of the proof of \cite[Theorem (3.2.3)]{kaCM}, every locally constant $R$-valued function $F$ supported on $\left(\OK\otimes\ZZ_p\right)^\times\times M_{n\times n}\left(\Oe\otimes\ZZ_p\right)$ that satisfies Equation \eqref{equnknualakacm} is an $R$-linear combination of $\OK$-valued functions $F$ supported on $\left(\OK\otimes\ZZ_p\right)^\times\times M_{n\times n}\left(\Oe\otimes\ZZ_p\right)$ that satisfy Equation \eqref{equnknualakacm}.  
So it is enough to prove the theorem for $\OK$-valued functions $F$.

Now, if we can construct an automorphic form satisfying the conditions of the theorem over $R=\IC$, then by the $q$-expansion principle \cite[Prop 7.1.2.15]{la}, the case over $R$ will follow for any $\OK$-subalgebra $R$ (in particular, for $R=\OK$) of $\IC$.  By \cite{lanalgan}, it sufficient to show that there is a $\IC$-valued $\ci$-automorphic form $G_{k, \nu, F}$ of weight $(k, \nu)$ holomorphic at $s=\frac{k}{2}$, whose Fourier coefficients (at $s=\frac{k}{2}$) are as in the statement of the theorem.  We will spend the remainder of this section (i.e. all of Section \ref{siegelchoice}) constructing such an automorphic form.
\end{proof}

\subsection{Construction of a $\ci$-automorphic form over $\IC$ whose Fourier coefficients meet the conditions of Theorem \ref{proplc}}\label{siegelchoice}
In this section, we construct the $\ci$-automorphic form $G_{k, \nu, F}$ necessary to complete the proof of Theorem \ref{proplc}.

Let $\mathfrak {m}$ be an ideal that divides $p^\infty$.  Let $\chi$ be a unitary Hecke character of type $A_0$
\begin{align*}
\chi: \adeles_K^\times\rightarrow \IC^\times
\end{align*}
of conductor $\mathfrak{m}$, i.e. 
\begin{align*}
\chi_v(a) = 1
\end{align*}
for all finite primes $v$ in $K$ and all $a\in K_v^\times$ such that
\begin{align*}
a\in 1+\mathfrak{m}_v{\OK}_v.
\end{align*}
Let $\nu(\sigma)$ and $k(\sigma)$, $\sigma\in\Sigma$, denote integers such that the infinity type of $\chi$ is
\begin{align}\label{infinitytypeknu}
\prod_{\sigma\in\Sigma}\sigma^{-k(\sigma)-2\nu(\sigma)}\left(\sigma\cdot\bar{\sigma}\right)^{\frac{k(\sigma)}{2}+\nu(\sigma)}.
\end{align}

For any $s\in \IC$, we view $\chi\cdot\left|\cdot\right|_K^{-s}\otimes |\cdot|_E^{-ns}$ as a character of the parabolic subgroup $GP_{+}(\adeles_E) = GM_{+}(\adeles_E)N(\adeles_E)\subseteq G_+\left(\adeles_E\right)$ via the composition of maps
\begin{tiny}
\begin{align*}
GP\left(\adeles_E\right)\xrightarrow{\mod N\left(\adeles_E\right)} GM\left(\adeles_E\right)&\xrightarrow{\mbox{map in \eqref{leviglv}}} \gl_{\adeles_K}(V\otimes_E\adeles_E)\times \gl_1\left(\adeles_E\right)\\
&\gl_{\adeles_K}(V\otimes_E\adeles_E)\times \gl_1\left(\adeles_E\right)
\xrightarrow{(h, \lambda)\mapsto|\lambda|_E^{-ns}\cdot\chi(\det h)\left|\det h\right|_K^{-s}}  \IC^\times.
\end{align*}
\end{tiny}
Consider the induced representation
\begin{align}\label{inducedrep}
I(\chi, s) &= \Ind_{GP_+(\adeles_E)}^{G_+(\adeles_E)} (\chi\cdot\left|\cdot\right|_K^{-s}\otimes |\omega(\cdot)|_K^{-ns/2})\nonumber\\&\cong \otimes_v\Ind_{GP_+(E_v)}^{G_+(E_v)}\left(\chi_v\cdot\left|\cdot\right|_v^{-2s}\otimes |\omega(\cdot)|_v^{-ns}\right),
\end{align}where the product is over all places of $E$.

Given a section $f\in I(\chi, s)$, the Siegel Eisenstein series associated to $f$ is the $\IC$-valued function of $G$ defined by
\begin{align*}
E_f(g) = \sum_{\gamma\in GP_+(E)\backslash G_+(E)}f(\gamma g)
\end{align*}
This function converges for $\Re(s)>0$ and can be continued meromorphically to the entire complex plane.
\begin{rmk}
As in \cite{AppToSHL}, if we were working with normalized induction, then the function would converge for $\Re(s)>\frac{n}{2}$, but we have absorbed the exponent $\frac{n}{2}$ into the exponent $s$. (Our choice not to include the modulus character at this point is equivalent to shifting the plane on which the function converges by $\frac{n}{2}$.)
\end{rmk}
All the poles of $E_f$ are simple and there are at most finitely many of them.  Details about the poles are given in \cite{tan}.  

As we noted in \cite[Section 2.2.4]{AppToSHL}, if the Siegel section $f$ factors as $f=\otimes_vf_v$, then $E_f$ has a Fourier expansion such that for all $h\in \gln(K)$ and $m\in\hern(K)$, 
\begin{align*}
E_f\left(\begin{pmatrix}1 & m\\ 0& 1\end{pmatrix}\begin{pmatrix}{ }^t\bar{h}^{-1} & 0\\ 0 & h\end{pmatrix}\right) = \sum_{\beta\in\hern(K)}c(\beta, h; f)\e_{\adeles_E}\left(\tr\left(\beta m\right)\right),
\end{align*}
with $c(\beta, h; f)$ a complex number dependent only on the choice of section $f$, the hermitian matrix $\beta\in \hern(K)$, $h_v$ for finite places $v$, and $\left(h\cdot  { }^t\bar{h}\right)_v$ for archimedean places $v$ of $E$.

By \cite[Sections 18.9, 18.10]{sh}, the Fourier coefficients of the Siegel sections $f=\otimes_vf_v$ that we will choose below are products of local Fourier coefficients determined by the local sections $f_v$.  More precisely, for each $\beta\in\hern(K)$,
\begin{align*}
c(\beta, h; f) &= C(n, K)\prod_vc_v(\beta, h; f),
\end{align*}
where
\begin{align}
c_v(\beta, h; f) &=\label{cvdef}\\
\int_{\hern(K\otimes E_v)}& f_v\left(\begin{pmatrix}0& -1\\ 1& 0\end{pmatrix}\begin{pmatrix}1 & m_v\\ 0& 1\end{pmatrix}\begin{pmatrix}{ }^t\bar{h_v}^{-1} & 0\\ 0 & h_v\end{pmatrix}\right)\e_v(-\tr(\beta_v m_v))dm_v,\nonumber\\
C(n, K)  &= 2^{n(n-1)[E:\IQ]/2}\left|D_E\right|^{-n/2}\left|D_K\right|^{-n(n-1)/4}\label{CnK},
\end{align}
$D_E$ and $D_K$ are the discriminants of $K$ and $E$ respectively, $\beta_v = \sigma_v(\beta)$ for each place $v$ of $E$, and $d_v$ denotes the Haar measure on $\hern(K_v)$ such that:
\begin{align}
&\int_{\hern\left(\OK\otimes_E E_v\right)} d_vx = 1, \mbox{ for each finite place $v$ of E}\nonumber\\
d_vx&: = \left|\bigwedge_{j=1}^{n}dx_{jj}\bigwedge_{j<k}\left(2^{-1}dx_{jk}\wedge d\bar{x}_{jk}\right)\right|, \mbox{ for each archimedean place $v$ of $E$}.\label{matrixnote}
\end{align}
(In Equation \eqref{matrixnote}, $x$ denotes the matrix whose $ij$-th entry is $x_{ij}$.)

Below, we recall \cite[Lemma 19]{AppToSHL}, which explains how the Fourier coefficients $c(\beta, h; f)$ transform when we change the point $h$.  For each $h\in \gln(\adeles_K)$ and $\lambda\in\adeles_E^\times$, let $m(h, \lambda)$ denote the matrix $\begin{pmatrix}{ }^t\bar{h}^{-1} & 0\\ 0 & \lambda h\end{pmatrix}$.  Generalizing Equation \eqref{cvdef}, we define
\begin{align*}
c_v(\beta, m(h,\lambda); f) &=
\int_{\hern(K\otimes E_v)}& f_v\left(\begin{pmatrix}0& -1\\ 1& 0\end{pmatrix}\begin{pmatrix}1 & m_v\\ 0& 1\end{pmatrix}m(h, \lambda)\right)\e_v(-\tr(\beta_v m_v))dm_v.
\end{align*}
We also define $c(\beta, m(h, \lambda); f) = C(n, K)\prod_vc_v(\beta, m(h, \lambda); f).$

\begin{lem}[Lemma 19 in \cite{AppToSHL}]\label{levieffect}
For each $h\in \gln(\adeles_K)$, $\lambda\in\adeles_E^\times$, and $\beta\in\hern(K)$,
\begin{align}\label{fcoeffsat1}
c&\left(\beta, \begin{pmatrix}{ }^t\bar{h}^{-1}& 0\\ 0& \lambda h\end{pmatrix}; f\right)\nonumber\\
& = \chi(\det\overline{\left(\lambda h\right)}^{-1})\left|\det\left(\overline{\left(\lambda h\right)}^{-1}\cdot\left(\lambda h\right)^{-1}\right)\right|_E^{n-s}|\lambda|_E^{-ns} c(\lambda^{-1}h^{-1}\beta{ }^t\bar{h}^{-1} , 1_n; f).
\end{align}
\end{lem}
\begin{proof}
Let $\eta = \begin{pmatrix}0 & -1_n\\ 1_n & 0\end{pmatrix}.$
Observe that for any $n\times n$ matrix $m$,
\begin{align*}
\eta \cdot m(h, \lambda)\cdot \eta^{-1} & = m(\lambda^{-1}{ }^t\bar{h}^{-1}, \lambda)\\
m(h, \lambda)^{-1}\cdot \begin{pmatrix}1&m\\ 0 &1\end{pmatrix}\cdot m(h, \lambda) &= \begin{pmatrix}1 & \lambda { }^t\bar{h} m h\\ 0 & 1\end{pmatrix}.
\end{align*}
Therefore,
\begin{align*}
\eta \cdot \begin{pmatrix}1 & m \\ 0 & 1\end{pmatrix}\cdot m(h, \lambda) & = \left(\eta \cdot m(h, \lambda)\cdot \eta^{-1}\right)\eta \left(m(h, \lambda)^{-1}\begin{pmatrix}1 & m\\ 0 & 1\end{pmatrix} m(h, \lambda)\right)\\
 & = m(\lambda^{-1}{ }^t\bar{h}^{-1}, \lambda) \eta \begin{pmatrix}1 & \lambda{ }^t\bar{h} m§ h\\ 0 & 1\end{pmatrix}.
\end{align*}
So for any place $v$ of $E$ and section $f_v\in \Ind_{GP(E_v)}^{G_+(E_v)}(\chi, s)$,
\begin{align}\label{leviunimp2}
f_v&\left(\eta\begin{pmatrix}1 & m \\ 0 & 1\end{pmatrix}m(h_v, \lambda)\right)\\
& = \chi_v\left(\det\overline{\lambda_vh_v}^{-1}\right)\left|\det\overline{\lambda_vh_v}^{-1}\right|^{-2s}_v|\lambda|_v^{-ns}f_v\left(\eta \begin{pmatrix}1 & \lambda { }^t\bar{h}_v m h_v\\ 0 & 1\end{pmatrix}\right). \end{align}
The lemma now follows from Equation \eqref{leviunimp2} and the fact that the Haar measure $d_v$ satisfies $d_v(\lambda h_vx{ }^t\bar{h_v}) = \left|\det\left(\lambda_v{ }^t\bar{h_v}\cdot h_v\right)\right|_v^{n}d_v(x)$ for each place $v$ of $E$. 
\end{proof}

So
\begin{align*}
c\left(\beta, \begin{pmatrix}\lambda^{-1}{ }^t\bar{h}^{-1}& 0\\ 0& h\end{pmatrix}; f\right) 
& = \chi\left(\lambda^n\right)\left|\lambda^{2n}\right|_E^{n-s}|\lambda|_E^{2ns}\left(\beta, \begin{pmatrix}{ }^t\bar{h}^{-1}& 0\\ 0& \lambda h\end{pmatrix}; f\right)\\
&=\left|\lambda^{2n^2}\right|_E\chi\left(\lambda^n\right)c\left(\beta, \begin{pmatrix}{ }^t\bar{h}^{-1}& 0\\ 0& \lambda h\end{pmatrix}; f\right).
\end{align*}

Below, we choose more specific Siegel sections $f = \otimes_vf_v$ and compute the corresponding Fourier coefficients. 
\subsubsection{The Siegel section at $\infty$}

In this section, we define a section $f_\infty^{k, \nu}=f_{\infty}^{k, \nu}\left(\bullet; i\cdot 1_n, \chi, s\right)\in\otimes_{v\divides\infty}\Ind_{GP_+(E_v)}^{G_+(E_v)}\left(\chi_v\cdot\left|\cdot\right|_v^{-2s}\otimes |\omega(\cdot)|_E^{-ns}\right)$.

For each $\alpha = \prod_{v\divides\infty}\alpha_v \in \prod_{v\divides\infty}G(E_v)$, we write $\alpha_v$ in the form $\begin{pmatrix}a_{v} & b_{v}\\ c_{v} & d_{v}\end{pmatrix}$ with $a_v, b_v, c_v,$ and $d_v$ $n\times n$ matrices.  Each element $\alpha\in G(E_v)$ acts on $z = \prod_{v\divides\infty}z_v\in\prod_{v\divides\infty}\hn$ by
\begin{align*}
\alpha_v\left(z_v\right) &= \left(a_vz_v+b_v\right)\left(c_vz_v+d_v\right)^{-1}\\
\alpha(z) & = \prod_{v\divides\infty}\alpha_v\left(z_v\right).
\end{align*}
Let
\begin{align*}
\lambda_{\alpha_v}\left(z_v\right)&= \lambda\left(\alpha_v, z_v\right) = \overline{c_v}\cdot{ }^tz_v+\overline{d_v}\\
\lambda_{\alpha}(z)& = \lambda(\alpha, z) = \prod_{v\divides\infty}\lambda_{\alpha_v}\left(z_v\right)\\
\mu_{\alpha_v}\left(z_v\right)&= \mu\left(\alpha_v, z_v\right) = c_v\cdot z_v+d_v\\
\mu_{\alpha}(z)& = \mu(\alpha, z) = \prod_{v\divides\infty}\mu_{\alpha_v}\left(z_v\right).
\end{align*}
(These are the canonical automorphy factors.  Properties of them are discussed in, for example, \cite[Section 3.3]{shar}.)  
We write
\begin{align*}
j_{\alpha_v}\left(z_v\right) & = j(\alpha_v, z_v) = \det\mu_{\alpha_v}\left(z_v\right)\\
j_{\alpha}(z) &= j(\alpha, z) = \prod_{v\divides\infty}j_{\alpha_v}\left(z_v\right).
\end{align*}
Note that
\begin{align}
\det\left(\lambda_{\alpha_v}\left(z_v\right)\right) &= \det\left(\overline{\alpha_v}\right)\omega\left(\alpha_v\right)^{-n}j_{\alpha_v}\left(z_v\right)\label{waystowritej1}\\
&= \det\left(\alpha_v\right)^{-1}\omega\left(\alpha_v\right)^nj_{\alpha_v}\left(z_v\right)\label{waystowritej2}.
\end{align}
So
\begin{align*}
\left|\det\left(\lambda_{\alpha_v}\left(z_v\right)\right)\right| = \left|j_{\alpha_v}\left(z_v\right)\right|.
\end{align*}
Consistent with the notation in \cite[Equation (10.4.3)]{sh}, we define
\begin{align*}
j_{\alpha}^{k, \nu}(z) &:= j_{\alpha}(z)^{k+\nu}\det\left(\lambda_{\alpha}(z)\right)^{-\nu}.
\end{align*}
By Equations \eqref{waystowritej1} and \eqref{waystowritej2}, we see that
\begin{align*}
j_{\alpha}^{k, \nu}(z)& = \left(\det\left(\overline{\alpha}\right)\omega(\alpha)^{-n}\right)^{-\nu}j_{\alpha}(z)^{k}\\
& = \left(\det\left(\alpha\right)^{-1}\omega(\alpha)^{n}\right)^{-\nu}j_{\alpha}(z)^{k}.
\end{align*}

Note that if $\beta= \prod_{v\divides\infty}\beta_v$ is also an element of $\prod_{v\divides\infty}G(E_v)$, then
\begin{align}
\lambda\left(\beta_v\alpha_v, z_v\right)& = \lambda\left(\beta_v, \alpha_vz_v\right)\lambda\left(\alpha_v, z_v\right)\label{bamult1}\\
\mu\left(\beta_v\alpha_v, z_v\right)& = \mu\left(\beta_v, \alpha_vz_v\right)\mu\left(\alpha_v, z_v\right).\label{bamult2}
\end{align}

Consistent with the notation in \cite[Section 3]{shar}, we define functions $\eta$ and $\delta$ on $\hn$ by
\begin{align*}
\eta(z) &= i\left({ }^t\bar{z}-z\right)\\
\delta(z) &= \det\left(\frac{1}{2}\eta(z)\right)
\end{align*}
for each $z\in\hn$.  So
\begin{align*}
\eta\left(i\cdot 1_n\right) &= 2\cdot 1_n\\
\delta\left(i\cdot 1_n\right) & = 1.
\end{align*}
We also write $\eta$ and $\delta$ to denote the functions $\prod_{\sigma\in\Sigma}\eta$ and $\prod_{\sigma\in\Sigma}\delta$, respectively, on $\prod_{\sigma\in\Sigma}\hn$.  So $\delta(\mathbf{i}) = 1$.  Also, note that
\begin{align*}
\delta\left(\alpha z\right) &= \omega(\alpha)^n\left|j_{\alpha}(z)\right|^{-2}\delta(z)\\
& = \omega(\alpha)^n\left|j_{\alpha}(z)\det\left(\lambda_{\alpha}(z)\right)\right|^{-1}\delta(z).
\end{align*}
Continuing to use the notation in \cite[Sections 3 and 5]{shar},
given $(k, \nu) = \prod_{v\divides\infty}\left(k_v, \nu_v\right)\in \left(\ZZ\times\ZZ\right)^\Sigma$, we define functions $f||_{k, \nu}$ and $f|_{k, \nu}$ on $\prod_{\sigma\in\Sigma}\hn$ by
\begin{align*}
\left(f||_{k, \nu}\alpha\right)(z) &= j_{\alpha}^{k, \nu}(z)^{-1}f(\alpha z)\\
f|_{k, \nu}\alpha& = f||_{k, \nu}\left(\omega(\alpha)^{-\frac{1}{2}}\alpha\right)
\end{align*}
for each $\IC$-valued function $f$ on $\hn$, point $z\in \hn$, and element $\alpha\in G$.  Note that $\omega(\alpha)^{-\frac{1}{2}}\alpha\in U(\eta_n)$, and if $\omega(\alpha_v) = 1$ for all $v\in\Sigma$, then
\begin{align*}
f|_{k, \nu}\alpha & = f||_{k, \nu}\alpha.
\end{align*}
More generally, for each function $f$ on $\prod_{\sigma\in\Sigma}\hn$ with values in some representation $(V, \rho)$ of $\prod_{\sigma\in\Sigma}\gln(\IC)\times\gln(\IC)$, we define functions $f||_{\rho}$ and $f|_{\rho}$ on $\hn$ by
\begin{align*}
\left(f||_{\rho}\alpha\right)(z) &= \rho\left(\mu_\alpha(z), \lambda_\alpha(z)\right)^{-1}f(\alpha z)\\
f|_{\rho}\alpha& = f||_{\rho}\left(\omega(\alpha)^{-\frac{1}{2}}\alpha\right).
\end{align*}
Note that we also use the notation $f||$ and $f|$ when we are working with just one copy of $\hn$, rather than $[E:\IQ]$ copies of $\hn$ at once.

We define 
\begin{align*}
f_{\infty}^{k, \nu} = \otimes f_v^{k, \nu}\in\left(\bullet; i\cdot 1_n, \chi, s\right)\in\otimes_{v\divides\infty}\Ind_{GP_+(E_v)}^{G_+(E_v)}\left(\chi_v\cdot\left|\cdot\right|_v^{-2s}\otimes |\omega(\cdot)|_E^{-ns}\right)
\end{align*}
by
\begin{align*}
f_{\infty}^{k,\nu}\left(\alpha; i\cdot 1_n, \chi, s\right) &= \left(\delta^{s-\frac{k}{2}}|_{k, \nu}\alpha\right)\left(i\cdot 1_n\right)\\
&=j_{\omega(\alpha)^{-1/2}\alpha}^{k, \nu}\left(i\cdot 1_n\right)^{-1}\left|j_{\omega(\alpha)^{-1/2}\alpha}\left(i\cdot 1_n\right)^{-2}\omega\left(\omega(\alpha)^{-1/2}\alpha\right)^{n}\right|^{s-\frac{k\left(\sigma_v\right)}{2}}\\
& = j_{\omega(\alpha)^{-1/2}\alpha}^{k, \nu}\left(i\cdot 1_n\right)^{-1}\left|j_{\omega(\alpha)^{-1/2}\alpha}\left(i\cdot 1_n\right)^{-2}\right|^{s-\frac{k}{2}}.
\end{align*}

Given $\alpha\in G$, we also define a function $f_{\infty}^{k, \nu}(\alpha; \bullet, \chi, s)$ on $\hn$ by
\begin{align*}
f_{\infty}^{k, \nu}(\alpha; z, \chi, s) & = \left(\delta^{s-\frac{k}{2}}|_{k, \nu}\alpha\right)\left(z\right)\\
&= j_{\omega(\alpha)^{-1/2}\alpha}^{k, \nu}(z)^{-1}\left|j_{\omega(\alpha)^{-1/2}\alpha}(z)^{-2}\right|^{s-\frac{k}{2}}\delta(z)^{s-\frac{k}{2}}.
\end{align*}
By Equations \eqref{bamult1} and \eqref{bamult2}, we see that if $g\in G$ is such that
\begin{align*}
g\left(\bf{i}\right) = z,
\end{align*}
then for each $\alpha\in G$,
\begin{align*}
f_{\infty}^{k, \nu}\left(\alpha g; i\cdot 1_n, \chi, s\right) &= f_{\infty}^{k, \nu}(\alpha; z, \chi, s)f_{\infty}^{k, \nu}\left(g; i\cdot 1_n, \chi, s\right)\delta(z)^{\frac{k}{2}-s}.
\end{align*}

For $k\in\ZZ$ and $\nu = \left(\nu_v\right)_{v\in\Sigma}\in\ZZ^\Sigma$, $f_{\infty}^{k, \nu}\left(\alpha; \bullet, \chi, s\right)$ is a holomorphic function on $\hn$ at $s=\frac{k}{2}$.

\subsubsection{The Fourier coefficients at archimedean places of $E$}\label{sectionsinfty}

When there is an integer $k$ such that 
\begin{align*}
s = \frac{k}{2} = \frac{k(\sigma)}{2}  \mbox{ for all $\sigma\in\Sigma$}
\end{align*}
(i.e. when $f_{\infty}^{k, \nu}\left(\alpha; z, \chi, s\right)$ is a holomorphic function of $z\in \hn$), \cite[Equation (7.12)]{sheseries} describes the archimedean Fourier coefficients precisely:
\begin{small}
\begin{align}
c_v&\left(\beta, 1_n; f_v^{k, \nu}\left(\bullet; i1_n, \chi, \frac{k}{2}\right)\right)\nonumber \\
 &= 2^{(1-n)n}i^{-nk}(2\pi)^{nk}\left(\pi^{n(n-1)/2}\prod_{t=0}^{n-1}\Gamma(k-t)\right)^{-1}\sigma_v(\det\beta)^{k-n}\e\left(i\tr(\sigma_v(\beta))\right)\label{coeffinfty},
\end{align}
\end{small}
for each archimedean place $v$ of $E$.  Observe that when $k\geq n$, 
\begin{align*}
\prod_{v\divides \infty}c_v\left(\beta, h; f_v^{k, \nu}\left(\bullet; i1_n, \chi, \frac{k}{2}\right)\right) = 0,
\end{align*}
unless $\det(\beta)\neq 0$ and $\det(h)\neq 0$, i.e. unless $\beta$ is of rank $n$.  Also, note that in our situation, $\beta$ will be in $\hern(K)$, so $\prod_{v\in \Sigma}\e\left(i\tr(\sigma_v(\beta))\right) = \e\left(i b\right)$ for some $b\in \IQ$, so $\prod_{v\in \Sigma}\e\left(i\tr(\sigma_v(\beta))\right) = \e\left(i b\right)$ is a root of unity.

\subsubsection{Siegel Sections at $p$}\label{sectionatp}

We work with Siegel sections at $p$ that are similar to the ones in \cite[Section 2.2.8]{AppToSHL}.  (To account for a similitude factor, we multiply the Siegel sections from \cite[Section 2.2.8]{AppToSHL} by $|\omega(g)|_p^{-ns}$.)

\begin{lem}[Lemma 10, \cite{AppToSHL}]\label{pfflemma}
Let $\Gamma$ be a compact open subset of $\prod_{v\in\Sigma}\gln({\Oe}_v)$, and let $\tilde{F}$ be a locally constant Schwartz function
\begin{align*}
\tilde{F}: \prod_{v\in\Sigma}\left(\Hom_{K_v}(V_v, V_{d, v})\oplus\Hom_{K_v}(V_v, V^d_v)\right)&\rightarrow R\\
(X_1, X_2)&\mapsto \tilde{F}(X_1, X_2)
\end{align*}
(with $R$ a subring of $\IC$) whose support in the first variable is $\Gamma$ and such that
\begin{align}\label{chi1chi2chi}
\tilde{F}\left(X, { }^tX^{-1}Y\right) = \prod_{v\in \Sigma}\chi_v\left(\det(X)\right)\tilde{F}(1, Y)
\end{align}
for all $X$ in $\Gamma$ and $Y$ in $\prod_{v\in\Sigma}M_{n\times n}(E_v)$.\footnote{The version of the righthand side of Equation \eqref{chi1chi2chi} appearing in \cite[Lemma 10]{AppToSHL} reads ``$\chi_1\chi_2^{-1}\left(\det(X)\right)F(1, Y)$.''  The characters denoted $\chi_1$ and $\chi_2$ in \cite{AppToSHL} have the property that $\chi_1\chi_2^{-1}(a) = \prod_{v\in\Sigma}\chi_v(a)$ for all $a\in \prod_{v\in\Sigma}{\Oe}_v$.  The function denoted by $\tilde{F}$ in the current paper is denoted by $F$ in \cite{AppToSHL}.}  There is a Siegel section $f^{P\tilde{F}(-X, Y)}$ at $p$ whose Fourier coefficient at $\beta\in M_{n\times n}(E_v)$
is
\begin{align*}
c(\beta, 1; f^{P\tilde{F}(-X, Y)}) = \mathrm{volume}(\Gamma)\cdot \tilde{F}\left(1, { }^t\beta\right).
\end{align*}
\end{lem}

Note that $PF$ stands for ``partial Fourier transform.''  We use that notation to be consistent with the notation in \cite[Section 2.2.8]{AppToSHL} (which was, in turn chosen to be consistent with the notation of \cite[Section 3.1]{kaCM}), but we do not need to discuss partial Fourier transforms here.

As a direct consequence of Lemma \ref{pfflemma}, we obtain the following corollary:
\begin{cor}\label{anyF}
For any locally constant Schwartz function $\tilde{F}$ satisfying the conditions of Lemma \ref{pfflemma} for some $\Gamma$ with positive volume, there is a Siegel section $f_{\tilde{F}}$ in $\otimes_{v\in\Sigma}\Ind_{P\left(E_v\right)}^{G\left(E_v\right)}\left(\chi_v\cdot\left|\cdot\right|^{-2s}\right)$ whose local (at $p$) Fourier coefficient at $\beta$ is $\tilde{F}\left(1, { }^t\beta\right)$.
\end{cor}

Furthermore, as we explain in Corollary \ref{reallyanyF}, we can significantly weaken the conditions placed on $\tilde{F}$ in Corollary \ref{anyF}.
\begin{cor}\label{reallyanyF}
Let $k$ be a positive integer.  Let $\tilde{F}$ be a locally constant Schwartz function
\begin{align*}
\tilde{F}: \left(\prod_{v\in\Sigma} \left(M_{n\times n}\left({\Oe}_v\right)\times M_{n\times n}\left({\Oe}_v\right)\right)\right)\rightarrow R
\end{align*}
whose support lies in $\prod_{v\in\Sigma}\left(\gln\left({\Oe}_v\right)\times M_{n\times n}\left({\Oe}_v\right)\right)$ and which satisfies
\begin{align*}
\tilde{F}(e, { }^te^{-1}y) = \mathbf{N}_{E/\IQ}(\det e)^k\tilde{F}(1, y),
\end{align*}
for all $e\in \gln\left(\Oe\right)$ contained in the support $\Gamma$ in the first variable of $\tilde{F}$.  Suppose, furthermore, that $\Gamma$ has positive volume.
Then there is a Siegel section $f_{\tilde{F}}\in \otimes_{v\in\Sigma}\Ind_{P\left(E_v\right)}^{G\left(E_v\right)}\left(\chi_v\cdot\left|\cdot\right|^{-2s}\right)$ whose local (at $p$) Fourier coefficient at $\beta$ is $\tilde{F}\left(1, { }^t\beta\right)$.
\end{cor}
\begin{proof}
Let $\tilde{F}$ be a locally constant Schwartz function
\begin{align*}
\tilde{F}: \prod_{v\in\Sigma} \left(M_{n\times n}\left({\Oe}_v\right)\times M_{n\times n}\left({\Oe}_v\right)\right)\rightarrow R
\end{align*}
whose support lies in $\prod_{v\in\Sigma}\left(\gln\left({\Oe}_v\right)\times M_{n\times n}\left({\Oe}_v\right)\right)$ and which satisfies
\begin{align}\label{katzfequ}
\tilde{F}(e, { }^te^{-1}y) = \mathbf{N}_{E/\IQ}(\det e)^k\tilde{F}(1, y),
\end{align}
for all $e\in \gln\left(\Oe\right)$ contained in the support in the first variable of $\tilde{F}$.  Then since $\tilde{F}$ is locally constant, has compact support, and satisfies Equation \eqref{katzfequ}, there is a unitary Hecke character $\chi$ whose infinity type is as in Expression \eqref{infinitytypeknu} and such that the conductor $\mathfrak{m} = p^d$ for $d$ a sufficiently large positive integer) so that
\begin{align*}
\tilde{F} = a_1F_1+\cdots + a_l F_l
\end{align*}
for some positive integer $l$ and $a_1, \ldots, a_l\in R$, and functions $F_1, \ldots, F_l$ meeting the conditions of Corollary \ref{anyF} (all for this {\it same character} $\chi$ but possibly with {\it different supports} $\Gamma_1, \ldots, \Gamma_l$, respectively, in the first variable).

Now, we define
\begin{align*}
f_{\tilde{F}} := a_1f_{F_1}+\cdots+a_lf_{F_l},
\end{align*}
where $f_{F_1}, \ldots, f_{F_l}$ are the Siegel sections obtained in Corollary \ref{anyF}.  Then $f_{\tilde{F}}$ is a linear combination of elements of the module $\otimes_{v\in\Sigma}\Ind_{P\left(E_v\right)}^{G\left(E_v\right)}\left(\chi_v\cdot\left|\cdot\right|^{-2s}\right)$.  So $f_{\tilde{F}}$ is itself an element of $\otimes_{v\in\Sigma}\Ind_{P\left(E_v\right)}^{G\left(E_v\right)}\left(\chi_v\cdot\left|\cdot\right|^{-2s}\right)$.  Now, the Fourier coefficient of a sum of Siegel sections is the sum of the Fourier coefficients of these Siegel sections.  So the Fourier coefficient at $\beta$ of $f_{\tilde{F}}$ is
\begin{align*}
a_1F_1(1, { }^t\beta)+\cdots+a_lF_l(1, { }^t\beta) = \tilde{F}(1, { }^t\beta).
\end{align*}
\end{proof}

\subsubsection{Siegel Sections away from $p$ and $\infty$}\label{sectionsnotpinfty}

We use the same Siegel sections at places $v\ndivides p\infty$ as in \cite[Section 2.2.9]{AppToSHL}.  We now recall the key properties of these Siegel sections, which are described in more detail in \cite[Section 18]{sh}.

Let $\mathfrak{b}$ be an ideal in $\Oe$ prime to $p$.  For each finite place $v$ prime to $p$, there is a Siegel section $f_v^{\mathfrak{b}} =f_v^{\mathfrak{b}}(\bullet; \chi_v, s)\in\Ind_{P(E_v)}^{G(E_v)}(\chi_v, s)$ with the following property: By \cite[Proposition 19.2]{sh}, whenever the Fourier coefficient $c(\beta, m(1); f_v^{\mathfrak{b}})$ is non-zero,
\begin{tiny}
\begin{align}\label{coeffnotpinfty}
\prod_{v\ndivides p\infty}c(\beta, m(1); f_v^{\mathfrak{b}}) = N_{E/\IQ}(\mathfrak{b}\Oe)^{-n^2}\prod_{i = 0}^{n-1}L^p\left(2s-i, \chi_{E}^{-1}\tau^i\right)^{-1}\prod_{v\ndivides p\infty}P_{\beta, v, \mathfrak{b}}\left(\chi_E(\pi_v)^{-1}\left|\pi_v\right|_v^{2s}\right),
\end{align}
\end{tiny}
where:
\begin{enumerate}
\item{the product is over primes of $E$;}
\item{the Hecke character $\chi_E$ is the restriction of $\chi$ to $E$;}
\item{the function $P_{\beta, v, \mathfrak{b}}$ is a polynomial that is dependent only on $\beta$, $v$, and $\mathfrak{b}$ and has coefficients in $\ZZ$ and constant term $1$;}
\item{the polynomial $P_{\beta, v, \mathfrak{b}}$ is identically $1$ for all but finitely many $v$;}
\item{$\tau$ is the Hecke character of $E$ corresponding to $K/E$;}
\item{$\pi_v$ is a uniformizer of $O_{E, v}$, viewed as an element of $K^\times$ prime to $p$;}
\item{\begin{align*}
L^p(r, \chi_E^{-1}\tau^i) = \prod_{v\ndivides p\infty\mathrm{cond}{\tau}}\left(1-\chi_v(\pi_v)^{-1}\tau^i(\pi_v)\left|\pi_v\right|_v^r\right)^{-1}.
\end{align*}
}
\end{enumerate}

\subsubsection{Global Fourier coefficients}

Recall that by Lemma \ref{levieffect}, the Fourier coefficients $c(\beta, h; f)$ are completely determined by the coefficients $c(\beta, 1_n; f)$.  In Proposition \ref{globalcoeffsprop}, we combine the results of Sections \ref{sectionsinfty}, \ref{sectionatp}, and \ref{sectionsnotpinfty} in order to give the global Fourier coefficients of the Eisenstein series $E_f$.

Let $\chi$ be a unitary Hecke character as above, and furthermore, suppose the infinity type of $\chi$ is
\begin{align}
\prod_{\sigma\in\Sigma}\sigma^{-k-2\nu(\sigma)}\left(\sigma\bar{\sigma}\right)^{\frac{k}{2}+\nu(\sigma)}\label{integralk}
\end{align}
(i.e. $k(\sigma) = k\in\ZZ$ for all $\sigma\in\Sigma$).  Let $C(n, K)$ be the constant dependent only upon $n$ and $K$ defined in Equation \eqref{CnK}.

\begin{prop}\label{globalcoeffsprop}  
Let $k\geq n$, let $\nu = \left(\nu(\sigma)\right)\in\ZZ^\Sigma$, and let
\begin{align}\label{aholosec}
f_{k, \nu, \chi, \tilde{F}} := f_{k, \nu, \chi, \mathfrak{b}, \tilde{F}} := \otimes_{v\in\Sigma}f_{\tilde{F},v}\otimes f^{k, \nu}_{\infty}\left(\bullet; i1_n, \chi, s\right)\otimes f^{\mathfrak{b}}\in \Ind_{P\left(\adeles_E\right)}^{G\left(\adeles_E\right)}\left(\chi\cdot\left|\cdot\right|_K^{-s} \right),
\end{align}
with $\chi$ as in Equation \eqref{integralk}, $\otimes_{v\divides\Sigma}f_{\tilde{F}, v}$ the section at $p$ from Corollary \ref{anyF}, $f^{k, \nu}_{\infty}$ the section at $\infty$ defined in Section \ref{sectionsinfty}, and $f^{\mathfrak{b}}$ the section away from $p$ and $\infty$ defined in Section \ref{sectionsnotpinfty}.
 
Then at $s=\frac{k}{2}$, all the nonzero Fourier coefficients $c(\beta, 1_n; f_{k, \nu, \chi, \tilde{F}})$ are given by
\begin{tiny}
\begin{align}\label{FourierCoeffFormula}
 D(n, K, \mathfrak{b}, p, k)\prod_{v\ndivides p\infty}P_{\beta, v, \mathfrak{b}}\left(\chi_E(\pi_v)^{-1}\left|\pi_v\right|_v^{k}\right)\tilde{F}\left(1, { }^t\beta\right)\prod_{v\in\Sigma}\sigma_v(\det\beta)^{k-n}\e\left(i\tr_{E/\IQ}(\beta)\right).
 \end{align}
\end{tiny}
where
\begin{tiny} 
\begin{align*}
D&(n, K, \mathfrak{b}, p, k)\\  
&= C(n, K)N(\mathfrak{b}{\Oe})^{-n^2}\prod_{i = 0}^{n-1}\left(2^{(1-n)n}i^{-nk}(2\pi)^{nk}\left(\pi^{n(n-1)/2}\prod_{t=0}^{n-1}\Gamma(k-t)\right)^{-1}\right)^{[E:\IQ]}\prod_{i = 0}^{n-1}L^p\left(k-i, \chi_{E}^{-1}\tau^i\right)^{-1}.
\end{align*}
\end{tiny}

\end{prop}

\begin{proof}
This follows directly from Equation \eqref{CnK}, Corollary \ref{anyF}, and Equations \eqref{coeffnotpinfty} and \eqref{coeffinfty}.  
\end{proof}

Given $\tilde{F}$ as above, define
\begin{align*}
\tilde{F}_{\chi}:\left(\OK\otimes\ZZ_p\right)\times M_{n\times n}\left(\Oe\otimes\ZZ_p\right)\rightarrow R
\end{align*}
to be the locally constant function whose support lies in
\begin{align*}
\left(\OK\otimes\ZZ_p\right)^\times\times M_{n\times n}\left(\Oe\otimes\ZZ_p\right)
\end{align*}
and which is defined on $\left(\OK\otimes\ZZ_p\right)^\times\times M_{n\times n}\left(\Oe\otimes\ZZ_p\right)$ by
\begin{align}\label{Fchiequ}
\tilde{F}_{\chi}\left(x, y\right)=\prod_{v\in\Sigma}\chi_v(x)\tilde{F}\left(1, \mathbf{N}_{K/E}(x){ }^ty\right),
\end{align}
where the product is over the primes in $\Sigma$ dividing $p$.  Then for all $e\in\OK^\times$,
\begin{align*}
\tilde{F}_{\chi}\left(ex, \mathbf{N}_{K/E}(e^{-1})y\right) = \mathbf{N}_{k, \nu}(e)\tilde{F}_{\chi}\left(x, y\right)
\end{align*}
for all $x\in \OK\otimes\ZZ_p$ and $y\in M_{n\times n}\left(\Oe\otimes\ZZ_p\right)$.  On the other hand, any locally constant function
\begin{align*}
F: \left(\OK\otimes\ZZ_p\right)\times M_{n\times n}\left(\Oe\otimes\ZZ_p\right)\rightarrow R
\end{align*}
supported on $\left(\OK\otimes\ZZ_p\right)^\times\times M_{n\times n}\left(\Oe\otimes\ZZ_p\right)$ which satisfies
\begin{align*}
F\left(ex, \mathbf{N}_{K/E}(e)^{-1}y\right) = \mathbf{N}_{E/\IQ}(e)^kF\left(x, y\right)
\end{align*}
for all $e\in\OK^\times$, $x\in \OK\otimes\ZZ_p$, and $y\in M_{n\times n}\left(\Oe\otimes\ZZ_p\right)$
can be written as a linear combination of such functions $\tilde{F}_{\chi}$ for Hecke characters $\chi$ of infinity type $(k, \nu)$ and conductor dividing $p^\infty$ and functions $\tilde{F}$ as above.

Now, let 
\begin{align*}
G_{k, \nu, \chi, \tilde{F}} = D(n, K, \mathfrak{b}, p, k)^{-1}E_{f_{k, \nu, \chi, \tilde{F}}}.
\end{align*}
Applying Proposition \ref{globalcoeffsprop}, we see that the Fourier coefficients of the holomorphic function $G_{k, \nu, \chi, \tilde{F}}\left(z, \frac{k}{2}\right)$ on $\hn$ are all finite $\ZZ$-linear combinations (over a finite set of $p$-adic units $a\in K$) of terms of the form
\begin{align}\label{aokoe}
\tilde{F}_{\chi}\left(a, \mathbf{N}_{K/E}(a)^{-1}\beta\right)\mathbf{N}_{k, \nu}\left(a^{-1}\det\beta\right)\mathbf{N}_{E/\IQ}\left(\det\beta\right)^{-n}
\end{align}
(We remark that although $\pi_v$ from Proposition \ref{globalcoeffsprop} is a place of $E$ for all $v$, the element $a$ from Expression \eqref{aokoe} might be in $K$ but not $\Oe$, depending on our choice of cusp.  The effect of the change of a cusp $m\in GM_+\left(\adeles_E\right)$, on $q$-expansions is given in Lemma \ref{levieffect}.)

Thus, we obtain the following result.
\begin{lem}\label{lemprop8}
Let $k\in\ZZ_{\geq n}$ and $\nu\in\ZZ^\Sigma$.  Let $F$ be a locally constant function
\begin{align*}
F: \left(\OK\otimes\ZZ_p\right)\times M_{n\times n}\left(\Oe\otimes\ZZ_p\right)\rightarrow R
\end{align*}
supported on $\left(\OK\otimes\ZZ_p\right)^\times\times M_{n\times n}\left(\Oe\otimes\ZZ_p\right)$ which satisfies
\begin{align*}
F\left(ex, \mathbf{N}_{K/E}(e)^{-1}y\right) = \mathbf{N}_{k, \nu}(e)F\left(x, y\right)
\end{align*}
for all $e\in\OK^\times$, $x\in \OK\otimes\ZZ_p$, and $y\in M_{n\times n}\left(\Oe\otimes\ZZ_p\right)$. 
Then there is a $\ci$-automorphic form $G_{k, \nu, F}(z, s)$ (on $U(n,n)$) of weight $(k, \nu)$  that is holomorphic at $s=k/2$ and whose Fourier expansion at $s=k/2$ at a cusp $m\in GM_+(\adeles_E)$ is of the form $\sum_{0<\beta\in L_m}c(\beta)q^\beta$ (where $L_{m}$ is the lattice in $\hern(K)$ determined by $m$), with $c(\beta)$ a finite $\ZZ$-linear combination of terms of the form given in Expression \eqref{aokoe}.
\end{lem}  
(We obtain $G_{k, \nu, F}$ by taking a linear combination of the automorphic forms $G_{k, \nu, \chi, \tilde{F}}$.)

\section{Differential Operators}\label{diffopssection}

\subsection{$C^\infty$ Differential Operators}

In this section, we summarize results on $\ci$-differential operators that were studied extensively by Shimura, for instance in \cite{sh84}, \cite{shclassical}, \cite[Section 23]{sh}, and \cite[Section 12]{shar}.  Let $T =M_{n\times n}(\IC)$; we identify $T$ with the tangent space of $\hn$.  For each nonnegative integer $d$, let $\mathfrak{S}_d(T)$ denote the vector space of $\IC$-valued homogeneous polynomial functions on $T$ of degree $d$.  (For instance, the $e$-th power of the determinant function $\det^e$ is in $\mathfrak{S}_{ne}(T)$.) We denote by $\tau^d$ the representation of $\gln(\IC)\times\gln(\IC)$ on $\mathfrak{S}_d(T)$ defined by
\begin{align*}
\tau^d\left(a, b\right)g (z) = g\left({ }^ta z b\right)
\end{align*}
for all $a, b\in \gln(\IC)$, $z\in T$, and $g\in \mathfrak{S}_d(T)$.  

The classification of the irreducible subspaces of polynomial representations of $\gln(\IC)$ and of irreducible subspaces of $\tau^r$ for each $r$ is provided in \cite[Section 2]{shclassical} and \cite[Sections 12.6 and 12.7]{sh}.  We summarize the key features needed for our results; further details can be found in those two references.  Given a matrix $a\in M_{n\times n}(\IC)$, let ${\det}_j(a)$\label{detidefn} denote the determinant of the upper left $j\times j$ submatrix of $a$.  Each polynomial representation of $\gln(\IC)$ can be composed into a direct sum of irreducible representations of $\gln(\IC)$.  Each irreducible representation $\rho$ of $\gln(\IC)$ contains a unique eigenvector $p$ of highest weight $r_1\geq \cdots r_n\geq 0$ (for a unique ordered $n$-tuple $r_1\geq \cdots \geq r_n\geq 0$ of integers dependent on $\rho$), which is a common eigenvector of the upper triangular matrices of $\gln(\IC)$ and satisfies
\begin{align}
\rho(a) p &= \prod_{j=1}^{n}{\det}_j(a)^{e_j} p\nonumber\\
e_j&= r_j-r_{j+1}, 1\leq j\leq n-1\label{ej1}\\
e_n & = r_n,\label{ej2}
\end{align}
for all $a$ in the subgroup of upper triangular matrices in $\gln(\IC)$.  Also, for each ordered $n$-tuple $r_1\geq \cdots \geq r_n\geq 0$, there is a unique corresponding irreducible polynomial representation of $\gln(\IC).$  If $\rho$ and $\sigma$ are irreducible representations of $\gln(\IC)$, then by\cite[Theorem 12.7]{shar}, $\rho\otimes\sigma$ occurs in $\tau^r$ if and only if the $\rho$ and $\sigma$ are representations of the same highest weights $r_1\geq \cdots \geq r_n$ as each other and $r_1+\cdots + r_n = r$.  In this case, $\rho\otimes\sigma$ occurs with multiplicity one in $\tau^r$, and the corresponding irreducible subspace of $\tau^r$ contains the polynomial $p(x) = \prod_{j=1}^n\det_j(x)^{e_j}$ (where $e_j$ is defined as in Equations \eqref{ej1} and \eqref{ej2}); this polynomial $p(x)$ is an eigenvector of highest weight with respect to both $\rho$ and $\sigma$.

Let $\left(Z, \tau_Z\right)$ be an irreducible subspace of $\left(\mathfrak{S}_d, \tau\right)$ of highest weight $r_1\geq \cdots \geq r_n$, and let $\zeta\in Z$.  By \cite{shclassical}, \cite[Section 23]{sh}, and \cite[Section 13]{shar}, there are $\ci$-differential operators $D_{k}\left(\zeta\right)$ that act on $\ci$-functions on $\hn$ and have the property that for all $\alpha\in U\left(\eta_n\right)$, $\zeta\in Z\subseteq \mathfrak{S}_d\left(T\right)$, and complex numbers $s$,
\begin{align}\label{diffopsgen}
D_{k}\left(\zeta\right)\left(\delta^s||_{k, \nu}\alpha\right) = i^d\psi_Z(-k-s)\left(\delta^s||_{k, \nu}\alpha\right)\cdot \zeta\left({ }^t\eta^{-1}{ }^t\overline{\lambda_{\alpha}}{ }^t\mu_{\alpha}^{-1}\right),
\end{align}
where (as proved in \cite[Theorem 4.1]{shclassical})
\begin{align*}
\psi_Z(s)=\prod_{h=1}^n\prod_{j=1}^{r_h}(s-j+h).
\end{align*}
If $\rho$ is the representation of $\gln(\IC)\times \gln(\IC)$, there is a differential operator $D_\rho^Z$ (defined on \cite[p. 222]{EDiffOps} and in \cite[Equation (12.20)]{shar}) such that for all $\ci$-functions $f$ on $\hn$, $D_\rho^Zf$ is a $\Hom(Z, \IC)$-valued $\ci$-function on $\hn$ with the property that
\begin{align}\label{gendifopeser}
\left(D_\rho^Z f\right)||_{\rho\otimes\tau_Z}\alpha = D_\rho^Z\left(f||_\rho \alpha\right)
\end{align}
for all $\alpha\in G$.  Furthermore, if $\rho$ is defined by $\rho(a, b) = \det(b)^k$, then as the proof of \cite[Lemma 23.4]{sh} explains, 
\begin{align*}
D_k(\zeta)f = (D_\rho^Z f)(\zeta).
\end{align*}
Note that when $Z$ is a $\Sigma$-tuple $\left(Z_v\right)_{v\in\Sigma}$, we also use the notation $\psi_Z$ to denote $\prod_{v\in \Sigma}\psi_{Z_v}$.

So for example, if $d\in\ZZ_{\geq 0}$ and $\zeta = \det^d$, then Equation \eqref{diffopsgen} becomes
\begin{align*}
D_{k}\left({\det}^d\right)\left(\delta^s||_{k, \nu}\alpha\right) &= i^{nd}\psi_Z(-k-s)\delta^s||_{k, \nu}\alpha\cdot {\det}^d\left({ }^t\eta^{-1}{ }^t\overline{\lambda_{\alpha}}\cdot { }^t\mu_{\alpha}^{-1}\right)\\
&= \left(\frac{i}{2}\right)^{nd}\prod_{h=1}^n\prod_{j=1}^d(-k-s-j+h)\delta^{s-d}||_{k+2d, \nu-d}\alpha.
\end{align*}  Consequently, if $d = \left(d(\sigma)\right)_{\sigma\in\Sigma}\in\ZZ^\Sigma_{\geq 0}$, then
\begin{multline*}
\left(\prod_{\sigma\in\Sigma}D_{k}\left({\det}^{d(\sigma)}\right)\right)\left(G_{k, \nu, F}\left(z,\frac{k}{2}\right)\right) \\
	= \prod_{\sigma\in\Sigma}\left(\frac{i}{2}\right)^{nd(\sigma)}\prod_{h=1}^n\prod_{j=1}^{d(\sigma)}(-k-j+h)G_{k+2d, \nu-d, F}\left(z, \frac{k}{2}\right),
\end{multline*}
as in \cite[Equation (43)]{AppToSHL}.

As noted in \cite[Section 6]{shclassical}, $G_{k, \nu, F}\left(z, s\right)$ is a special case of the automorphic form $G_{k, \nu, \zeta, F}\left(z, s\right)$ that satisfies
\begin{align*}
D_k(\zeta)\left(G_{k, \nu, F}\left(z, \frac{k}{2}\right)\right) = \prod_{v\in\Sigma}i^{d_v}\psi_{Z_v}(-k)G_{k, \nu, \zeta, F}\left(z, \frac{k}{2}\right),
\end{align*}
where
\begin{align*}
D_k(\zeta) = \prod_{v\in\Sigma}D_k(\zeta_v).
\end{align*}
The case where $\zeta$ is a highest weight vector will be of particular interest to us.

\subsection{Rational Representations}\label{rationalrepresentations-section}

In order to generalize our discussion from the $\ci$-setting to the $p$-adic setting, we introduce rational representations, following \cite[Section 8.1.2]{hida} (which, in turn, summarizes relevant results from \cite{hidaGME} and \cite{jantzenRAG}).  

Let $A$ be a ring or a sheaf of rings over a scheme.  
Let $B$ denote the Borel subgroup of $\gln$ consisting of upper triangular matrices in $\gln$.  Let $N$ denote the unipotent radical of $B$.  Let $T\cong B/N$ denote the torus.  Following the notation of \cite[Section 8.1.2]{hida}, for each character $\kappa$ of $T$, we define\begin{align*}
R_A[\kappa]&=\Ind_{B}^{\gln}(\kappa)\\
&\left(= \left\{f:\gl_n/N\rightarrow \mathbf{A}^1\mid
 f(ht) = \kappa(t)f(h) \mbox{ for all } t\in T, h\in \gln/N\right\}\right).
\end{align*}
The group $\gln$ acts on $R_A[\kappa]$ via
\begin{align*}
(g\cdot f)(x) = f\left(g^{-1} x\right).
\end{align*}

As noted on \cite[p. 332]{hida}, there is a unique (up to $A$-unit multiple) $N$-invariant linear form $\ellcan$ in the dual space  $R_A[\kappa]\dual$ that generates $\left(R_A[\kappa]\dual\right)^N$ and can be normalized so that for all $f\in R_A[\kappa]$
\begin{align*}
\ellcan (f) = f\left(1_n\right),
\end{align*}
where $1_n$ denotes the origin in $\gln/N$.

Note that for each $\ci$-automorphic form $f$ on $\prod_{v\in\Sigma}\hn$ such that $f||_{k, \nu}\alpha = f$ (for all $\alpha$ in some congruence subgroup) and each highest weight vector $\zeta$ in an irreducible representation of highest weight $\kappa$, we may view $D_k(\zeta)f$ as an $R_{\IC}[\det^{k+\nu}\cdot\kappa]\otimes R_{\IC}[\det^{-\nu}\cdot\kappa]$-valued function on $\hn$.
We define a corresponding character $\kappa_{k, \nu}\left(t_1, \ldots, t_n, t_{n+1}, \ldots, t_{2n}\right) = \prod_{i=1}^nt_i^{k+\nu}t_{i+n}^{-\nu}$ on $T(\IC)\times T(\IC)$.

\subsection{The Algebraic Geometric Setting}

As explained in detail in \cite[Section 8.4]{EDiffOps}, which generalizes \cite[Section 2.3]{kaCM}, the $\ci$-differential operators discussed by Shimura have a geometric interpretation in terms of the Gauss-Manin connection.  $\ci$-automorphic forms can, as explained in \cite[Section 2]{EDiffOps}, be interpreted as sections of a vector bundle on (the complex analytification of) the moduli spaces $\mathcal{M}_{n,n}=\Sh(W)$.  Applying a differential operator (as discussed in \cite[Sections 6-9]{EDiffOps}) to an automorphic form of weight $\rho$ on $\mathcal{M}_{n,n}$ sends it to an automorphic form of weight $\rho\otimes \tau$ on $\mathcal{M}_{n,n}$.

We now recall the setting of \cite[Section 3]{AppToSHL}, as we will momentarily be in a similar (but not identical) situation.  For any $\OK$-algebra $R$, the $R$-valued points of $_\mathcal{K}\Sh(R)$ parametrize tuples $\uA$ consisting of an abelian variety together with a polarization, endomorphism, and level structure.  (We shall not need further details of these points here; more details are available, though, in a number of references, including \cite[Chapter 1]{la}, \cite[Chapter 7]{hida}, \cite[Section 2]{EDiffOps}.)  Given a point $\uA$ in $_\mathcal{K}\Sh(R)$, we write $\uo_{\uA/R} = \uo_{\uA/R}^+\oplus\uo_{\uA/R}^-$ for the sheaf of one-forms on $\uA$.  (As in \cite[Section 2]{EDiffOps}, $\uo_{\uA/R}^+$ and $\uo_{\uA/R}^-$ are the rank $n$ submodules determined by the action of $\OK$.)  
We identify $G(\IQ)\backslash X\times G(\adeles_f)/\mathcal{K}$ (which we identify with copies of $\hn$) with the points of $_{\mathcal{K}}\Sh(\IC)$; we shall write $\uA(z)$ to mean the point of $\uA$ identified with $z\in \prod_{v\in\Sigma}\hn$ under this identification.  Under this identification, if we fix an ordered basis of differentials  $u_1^\pm, \ldots, u_n^\pm$ for $\uo_{\A/\hn}^\pm$, then an automorphic form $f$ on $\hn$ corresponds to an automorphic form $\tilde{f}$ on $_{\mathcal{K}}\Sh(\IC)$ via
\begin{align*}
f(z) = \tilde{f}(\uA(z), u_1^\pm(z), \ldots, u_n^\pm(z)),
\end{align*}
Any other ordered basis of differentials for $\uo_{\uA/\IC}^\pm$ is simply obtained by the linear action of $\gln(\OK\otimes \IC) \cong \gln(\IC)\times\gln(\IC)$ on $\uo(z)= \uo(z)^+\oplus\uo(z)^-$, and 
\begin{align*}
\tilde{f}\left(\uA(z), g\cdot(u_1^\pm(z), \ldots, u_n^\pm(z))\right) = g\cdot \left(f\left(\uA(z), u_1^\pm(z), \ldots, u_n^\pm(z)\right)\right)
\end{align*}

\subsubsection{A $p$-adic analogue}In \cite[Section 9]{EDiffOps}, we discussed a $p$-adic analogue $\theta_\rho^Z$ of the differential operators $D_\rho^Z$.  The differential operators $\theta_\rho^Z$ act on sections of certain vector bundles on the Igusa tower $T_{\infty, \infty}$ (a formal scheme over the ordinary locus of $_\mathcal{K}\Sh(R)$, for $R$ a mixed characteristic discrete valuation ring with residue characteristic $p$); for details on the Igusa tower, see \cite[Section 8]{hida}.  More precisely, $\theta_\rho^Z$ acts on sections of $R_{T_{\infty, \infty}}[\kappa]$ for various weights $\kappa$.  Note that by \cite[map (8.4)]{hida}, 
\begin{align}\label{ellcanh0}
\ellcan: H^0\left(T_{\infty, \infty}, R_{T_{\infty, \infty}}[\kappa]\right)\rightarrow V^N[\kappa]
\end{align}
 is an injective map into the space $V^N[\kappa]=V_{\infty, \infty}^N[\kappa]$ of $p$-adic modular forms of weight $\kappa$.  Given a highest weight vector $\zeta$ in $Z$, we define $\theta(\zeta) :=\theta_k:=\ellcan\circ \theta_\rho^Z$, where $\rho(a, b):=\det(b)^k$.


In \cite[Section 9]{EDiffOps}, we gave a formula for the action of $p$-adic differential operators $\theta_\rho^Z$ on $q$-expansions.  In particular, if the $q$-expansion of a scalar weight form $f\in H^0\left(T_{\infty, \infty}, R_{T_{\infty, \infty}}[\kappa]\right)$ at a cusp $m\in GM$ is
\begin{align*}
f(q) = \sum_\beta a(\beta) q^\beta,
\end{align*}
and $\zeta$ is a highest weight vector, then it follows from the formulas in \cite[Section 9]{EDiffOps} that
\begin{align}\label{thetaction}
\left(\theta(\zeta) f\right)(q) = \sum_\beta a(\beta)\cdot\zeta(\beta)q^\beta.
\end{align}

\section{A $p$-adic Eisenstein Measure with values in the space of Vector-Weight Automorphic forms}\label{emeasuresection}

\subsection{$p$-adic Eisenstein Series}\label{emeasuressection1}

As we explain in Theorem \ref{thmcts}, when $R$ is a (profinite) $p$-adic ring, we can extend Theorem \ref{proplc} to the case of continuous (not necessarily locally constant) functions $F$.  For the remainder of the paper, let $N$ be as in Section \ref{rationalrepresentations-section}. 
\begin{thm}\label{thmcts}
Let $R$ be a (profinite) $p$-adic $\OK$-algebra.  Fix an integer $k\geq n$, and let $\nu = \left(\nu(\sigma)\right)_{\sigma\in\Sigma}\in\ZZ^\Sigma$.  Let
\begin{align*}
F: \left(\OK\otimes\ZZ_p\right)\times M_{n\times n}\left(\Oe\otimes\ZZ_p\right)\rightarrow R
\end{align*}
be a continuous function supported on $\left(\OK\otimes\ZZ_p\right)^\times\times\gln\left(\Oe\otimes\ZZ_p\right)$ which satisfies
\begin{align*}
F\left(ex, \mathbf{N}_{K/E}(e)^{-1}y\right) = \mathbf{N}_{k, \nu}(e)F\left(x, y\right)
\end{align*}
for all $e\in \OK^\times$, $x\in \OK\otimes\ZZ_p$ and $y\in \gln\left(\Oe\otimes\ZZ_p\right)$.  Then there exists a $p$-adic automorphic form $G_{k, \nu, F}$ whose $q$-expansion at a cusp $m\in GM$ is of the form $\sum_{0<\beta\in L_m}c(\beta)q^\beta$ (where $L_{m}$ is the lattice in $\hern(K)$ determined by $m$), with $c(\beta)$ a finite $\ZZ$-linear combination of terms of the form
\begin{align*}
F\left(a, \mathbf{N}_{K/E}(a)^{-1}\beta\right)\mathbf{N}_{k, \nu}\left(a^{-1}\det\beta\right)\mathbf{N}_{E/\IQ}\left(\det\beta\right)^{-n}
\end{align*}
(where the linear combination is the sum over a finite set of $p$-adic units $a\in K$ dependent upon $\beta$ and the choice of cusp $m\in GM$).
\end{thm}
\begin{proof}
The proof is similar to the the proof of \cite[Theorem (3.4.1)]{kaCM}.  We remind the reader of the idea of \cite[Theorem (3.4.1)]{kaCM}.  For each integer $j\geq 1$, define
\begin{align*}
F_j: \left(\OK\otimes\ZZ_p\right)\times M_{n\times n}\left(\Oe\otimes\ZZ_p\right)&\rightarrow R/p^jR\\
F_j(x, y) = F(x, y)\mod p^jR.
\end{align*}
Then $F_j$ is a locally constant function satisfying the conditions of Theorem \ref{proplc}.  So by the $q$-expansion principle for $p$-adic forms (\cite[Corollary 10.4]{hi05}, \cite[Section 8.4]{hida}), there is a $p$-adic automorphic form $G_{k, \nu, F}$ whose $q$-expansion satisfies the conditions in the statement of the theorem.
\end{proof}

\begin{cor}\label{corcts}
Let $R$ be a (profinite) $p$-adic $\OK$-algebra, let $\nu = \left(\nu(\sigma)\right)_{\sigma\in\Sigma}\in\ZZ^\Sigma$, and let $k\geq n$ be an integer.  Let
\begin{align*}
F: \left(\OK\otimes\ZZ_p\right)\times M_{n\times n}\left(\Oe\otimes\ZZ_p\right)\rightarrow R
\end{align*}
be a continuous function supported on $\left(\OK\otimes\ZZ_p\right)^\times\times \gln\left(\Oe\otimes\ZZ_p\right)$ which satisfies
\begin{align*}
F\left(ex, \mathbf{N}_{K/E}(e)^{-1}yz\right) = \mathbf{N}_{k, \nu}(e)F\left(x, y\right)
\end{align*}
for all $e\in \OK^\times$, $x\in \OK\otimes\ZZ_p$ and $y\in M_{n\times n}\left(\Oe\otimes\ZZ_p\right)$.  Then
\begin{align}\label{relna}
G_{k,\nu, F} = G_{n, 0, \mathbf{N}_{k-n, \nu}\left(x^{-1}\mathbf{N}_{K/E}(x)^n\det y\right)F(x, y)},
\end{align}
where
\begin{align*}
\mathbf{N}_{k-n, \nu}\left(x^{-1}\mathbf{N}_{K/E}(x)^n\det y\right)F(x, y),
\end{align*}
denotes the function defined by 
\begin{align*}
(x, y)\mapsto \mathbf{N}_{k-n, \nu}\left(x^{-1}\mathbf{N}_{K/E}(x)^n\det y\right)F(x, y).
\end{align*}
on $\left(\OK\otimes\ZZ_p\right)^\times\times M_{n\times n}\left(\Oe\otimes\ZZ_p\right)$ and extended by $0$ to all of $\left(\OK\otimes\ZZ_p\right)\times M_{n\times n}\left(\Oe\otimes\ZZ_p\right)$.
\end{cor}
\begin{proof}
This follows from the $q$-expansion principle \cite[Corollary 10.4]{hi05}.
\end{proof}

\begin{rmk}
We comment now on the relationship between the weight of $G_{n, 0, F}$ and the $p$-adically continuous function $F$ appearing in the subscript.  By Corollary \ref{corcts} and Theorem \ref{proplc}, we have that if $F$ is a locally constant function satisfying the conditions of Corollary \ref{corcts}, then the $p$-adic automorphic form $G_{n, 0, \mathbf{N}_{k-n, \nu}\left(x^{-1}\mathbf{N}_{K/E}(x)^n\det y\right)F(x, y)}$ is the weight $(k, \nu)$ $p$-adic automorphic form $G_{k, \nu, F}$.  More generally, by Equation \eqref{thetaction}, the $p$-adic automorphic form $G_{n, 0, \mathbf{N}_{k-n, \nu}\left(x^{-1}\mathbf{N}_{K/E}(x)^n\det y\right)F(x, y)\zeta_\kappa(\mathbf{N}_{K/E}(x)y^{-1})}$ is the weight $\kappa\cdot\kappa_{k, \nu}$ $p$-adic automorphic form $\theta(\zeta_\kappa)G_{k, \nu, F}$, where $\zeta_\kappa$ is a highest weight vector for the representation of weight $\kappa$.  In particular, the $p$-adic automorphic form $G_{n, 0, \mathbf{N}_{k-n, \nu}\left(x^{-1}\mathbf{N}_{K/E}(x)^n\det y\right)F(x, y)\det(\mathbf{N}_{K/E}(x)y^{-1})^d}$ is the  $p$-adic automorphic form $\theta(\det^d)G_{n, 0, \mathbf{N}_{k-n, \nu}\left(x^{-1}\mathbf{N}_{K/E}(x)^n\det y\right)F(x, y)\zeta_\kappa(\mathbf{N}_{K/E}(x)y^{-1})}$ of weight $(k+2d, \nu-d)$.
\end{rmk}


\subsubsection{CM Points and Pullbacks}

In this section, we compare the values of certain $p$-adic automorphic forms and $\ci$-automorphic forms at CM points.\footnote{The significance of CM points is that they correspond to points of $U(n)\times U(n)\subseteq U(n,n)$, which are the points used (for instance, by Shimura) to study algebraicity of values of Eisenstein series, which are used in turn to study algebraicity of values of certain $L$-functions (through the doubling method, or ``pull back method," a construction of $L$-functions described in various sources, including \cite[Part A]{GPSR} and \cite[Section 2]{co}).  For the reader who is unfamiliar with this area, we note that determining the {\it precise} values of these Eisenstein series at CM points is neither necessary nor generally computationally feasible at this time.}  This material extends \cite[Section 3.0.1]{AppToSHL} beyond the case of scalar weights.  Let $R$ be an $\OK$-subalgebra of $\overline{\IQ}\cap\iota_{\infty}^{-1}\left(\mathcal{O}_{\IC_p}\right)$ in which $p$ splits completely.  Note that the embeddings $\iota_\infty$ and $\iota_p$ restrict to $R$ to give embeddings
\begin{align*}
\iota_{\infty}: R&\hookrightarrow \IC\\
\iota_p: R&\hookrightarrow R_0 = \varprojlim_mR/p^mR.
\end{align*}
Let $\uA$ be a CM abelian variety with PEL structure over $R$ (i.e. a CM point of the moduli space $_K\Sh(R)$, or equivalently, a point of\ $Sh(U(n)\times U(n))\hookrightarrow Sh(U(n,n))$).  Note that by extending scalars, we may also view $\uA$ as an abelian variety over $\IC$ or $R_0$.  

By an argument similar to that in \cite[Section 3.0.1]{AppToSHL}, there are complex and $p$-adic periods $\Omega = \left(\Omega^+, \Omega^-\right)\in \left(\IC^\times\right)^n\times\left(\IC^\times\right)^n$ and $c =\left(c^+, c^-\right)\in \left(\mathcal{O}_{\IC_p}^\times\right)^n\times\left(\mathcal{O}_{\IC_p}^\times\right)^n$, respectively, attached to each CM abelian variety $\uA$ over $R$ such that (if $F$ is $R$-valued, so $G_{k, \nu, F}$ arises over $R$)
\begin{align}\label{AlgCMequ}
\left(\kappa\cdot\kappa_{k, \nu}\right)^{-1}(\Omega)&\prod_{\sigma\in\Sigma}\kappa_\sigma(2\pi i)\psi_{Z}(-k) G_{k, \nu, \zeta, F}\left(z; h, \chi, \mu, \frac{k}{2}\right)\nonumber \\
&=\left(\kappa\cdot\kappa_{k, \nu}\right)^{-1}(c)\theta(\zeta)G_{k, \nu, F}(\uA),
\end{align}
where $z$ is a point in $\prod_{\sigma\in\Sigma}\hn$ corresponding to the CM abelian variety $\uA$ viewed as an abelian variety over $\IC$ (by extending scalars to $\IC$).  Here, $Z$ is the irreducible subrepresentation of $\prod_{v\in\Sigma}\gln(\IC)\times\gln(\IC)$ of highest weight $\kappa\in (\ZZ^n)^\Sigma$ and has $\zeta$ as a highest weight vector; by $\kappa(a)$ with $a$ a scalar, we mean $\kappa$ evaluated at the $n$-tuple $(a, \ldots, a)$ in the torus.
(The periods $\Omega$ and $c$ can be defined uniformly for all CM points at once, as explained in \cite[Section 5.1]{kaCM}.  For the present paper, though, this is not necessary.)  Note that when $\kappa  = \det^d$ (i.e. is the highest weight for a $1$-dimensional representation), we recover \cite[Equation (45)]{AppToSHL}.

\subsection{Eisenstein Measures}\label{emeasuressection2}

In analogue with \cite[Lemma (4.2.0)]{kaCM} (which handles the case of Hilbert modular forms), we have the following lemma (which applies to all integers $n\geq 1$).
\begin{lem}\label{HFLem}  Let $R$ be a $p$-adic $\OK$-algebra.  Then the inverse constructions
\begin{align}
H\left(x, y\right) & = \frac{1}{\mathbf{N}_{n, 0}\left(x\mathbf{N}_{K/E}(x)^{-n}\det y\right)}F\left(x, y^{-1}\right)\label{HF1}\\
F\left(x, y\right) & = \frac{1}{\mathbf{N}_{n, 0}\left(x^{-1}\mathbf{N}_{K/E}(x)^{n}\det y\right)}H\left(x, y^{-1}\right)\label{HF2}
\end{align}
give an $R$-linear bijection between the set of continuous $R$-valued functions
\begin{align*}
F: \left(\OK\otimes\ZZ_p\right)^\times\times\gln\left(\Oe\otimes\ZZ_p\right)\rightarrow R
\end{align*}
 satisfying
 \begin{align*}
 F\left(ex, \mathbf{N}_{K/E}(e)^{-1} y\right) = \mathbf{N}_{n, 0}\left(e\right)F\left(x, y\right)
 \end{align*}
 for all $e\in\OK^\times$ and the set of continuous $R$-valued functions
 \begin{align*}
 H: \left(\OK\otimes\ZZ_p\right)^\times\times \gln\left(\Oe\otimes\ZZ_p\right)\rightarrow R
 \end{align*}
  satisfying
 \begin{align*}
H\left(ex, \mathbf{N}_{K/E}(e)y\right) = H\left(x, y\right)
 \end{align*}
 for all $e\in\OK^\times$.
\end{lem}
\begin{proof}
The proof follows immediately from the properties of $F$ and $H$.
\end{proof}

Let 
\begin{align}\label{equgndefn}
\mathcal{G}_n =  \left(\left(\OK\otimes\ZZ_p\right)^\times\times \gln\left(\Oe\otimes\ZZ_p\right)\right)/\overline{\OK^\times},
\end{align}
 where $\overline{\OK^\times}$ denotes the $p$-adic closure of $\OK^\times$ embedded diagonally (as $(e, \mathbf{N}_{K/E}(e))$) in $\left(\OK\otimes\ZZ_p\right)^\times\times\gln\left(\Oe\otimes\ZZ_p\right)$ (and as before, $\left(\Oe\otimes\ZZ_p\right)^\times$ is embedded diagonally inside of $\gln\left(\Oe\otimes\ZZ_p\right)$).  Then Lemma \ref{HFLem} gives a bijection between the $R$-valued continuous functions $H$ on $\mathcal{G}_n$ and the $R$-valued continuous functions $F$ on $\left(\OK\otimes\ZZ_p\right)^\times\times\gln\left(\Oe\otimes\ZZ_p\right)$ satisfying  $F\left(ex, \mathbf{N}_{K/E}(e)^{-1} y\right) = \mathbf{N}_{n, 0}\left(e\right)F\left(x, y\right)$ for all $e\in\OK^\times$.
 
For any (profinite) $p$-adic ring $R$, an $R$-valued $p$-adic measure on a (profinite) compact, totally disconnected topological space $Y$ is a $\ZZ_p$-linear map
\begin{align*}
\mu: \mathcal{C}(Y, \ZZ_p)\rightarrow R,
\end{align*}
or equivalently (as explained in \cite[Section 4.0]{kaCM}), an $R'$-linear map
\begin{align*}
\mu: \mathcal{C}(Y, R')\rightarrow R
\end{align*}
for any $p$-adic ring $R'$ such that $R$ is an $R'$-algebra.  Instead of $\mu(f)$, one typically writes
\begin{align*}
\int_{Y}f d\mu.
\end{align*}
In Theorem \ref{emeasurethm}, we specialize to the case where $R$ is the ring $\Vnn$ of $p$-adic automorphic forms on $U(n,n)$ and $Y$ is the group $\mathcal{G}_n$ defined in Equation \eqref{equgndefn}.

\begin{thm}[A $p$-adic Eisenstein Measure for Vector-Weight Automorphic Forms]\label{emeasurethm}
Let $R$ be a profinite $p$-adic ring.  There is a $\Vnn$-valued $p$-adic measure $\mu = \mu_{\mathfrak{b}, n}$ on $\mathcal{G}_n$ defined by
\begin{align*}
\int_{\mathcal{G}_n}H d\mu_{\mathfrak{b}, n} = G_{n, 0, F}
\end{align*}
for all continuous $R$-valued functions $H$ on $\mathcal{G}_n$, with
\begin{align*}
F\left(x, y\right) & = \frac{1}{\mathbf{N}_{n, 0}\left(x^{-1}\mathbf{N}_{K/E}(x)^{n}\det y\right)}H\left(x, y^{-1}\right)
\end{align*}
extended by $0$ to all of $\left(\OK\otimes\ZZ_p\right)\times M_{n\times n}\left(\Oe\otimes\ZZ_p\right)$.
\end{thm}
\begin{proof}
Note that $F$ is the function corresponding to $H$ under the bijection in Lemma \ref{HFLem}.  The theorem then follows immediately from Theorem \ref{thmcts}, Corollary \ref{corcts}, Lemma \ref{HFLem}, and the $q$-expansion principle. 
\end{proof}
Note that the measure $\mu_{\mathfrak{b}, n}$ depends only upon $n$ and $\mathfrak{b}$.  In Section \ref{symplectickatz}, we relate the measure $\mu_{\mathfrak{b}, n}$ to the Eisenstein measure in \cite[Definition (4.2.5) and Equation (5.5.7)]{kaCM} and comment on how $\mu_{\mathfrak{b}, n}$ can be modified to the case of Siegel modular forms (i.e. automorphic forms on symplectic groups).

It follows from the definition of the measure $\mu_{\mathfrak{b}, n}$ in Theorem \ref{emeasurethm} that for each highest weight vector $\zeta_\kappa$ of highest weight $\kappa$,
\begin{align*}
\int_{\mathcal{G}_n}H(x, y)\zeta_\kappa\left(\mathbf{N}_{K/E}\left(x\right)y^{-1}\right) d\mu_{\mathfrak{b}, n} &= \theta\left(\zeta_\kappa\right) G_{n, 0, F(x, y)}.
\end{align*}
Now, let $\uA$ be an ordinary CM abelian variety with PEL structure over a subring $R$ of $\overline{\IQ}\cap\OCp$ (i.e. a CM point of the moduli space $_K\Sh(R)$, or equivalently, a point of $Sh(U(n)\times U(n))\hookrightarrow Sh(U(n,n))$).  As discussed above, by extending scalars, we may also view $\uA$ as an abelian variety over $\IC$ or over $R_0=\varprojlim_mR/p^mR$.  It follows from Equation \eqref{AlgCMequ} and Corollary \ref{corcts} that for $F(x,y)$ locally constant, supported on $\left(\OK\otimes\ZZ_p\right)^\times\times \gln\left(\Oe\otimes\ZZ_p\right)$, and satisfying
\begin{align*}
F\left(ex, \mathbf{N}_{K/E}(e)^{-1}y\right) = \mathbf{N}_{k, \nu}(e)F\left(x, y\right)
\end{align*}
for all $e\in \OK^\times$, $x\in \OK\otimes\ZZ_p$ and $y\in \gln\left(\Oe\otimes\ZZ_p\right)$,
\begin{align}
\left(\kappa\cdot\kappa_{k, \nu}\right)^{-1}&(c)\left(\int_{\mathcal{G}_n}\frac{1}{\mathbf{N}_{k, \nu}\left(x\mathbf{N}_{K/E}(x)^{-n}\det y\right)}F\left(x, y^{-1}\right)\zeta_\kappa\left(\mathbf{N}_{K/E}\left(x\right)y^{-1}\right) d\mu_{\mathfrak{b}, n}\right)(\uA)\\
&=\left(\kappa\cdot\kappa_{k, \nu}\right)^{-1}(\Omega)\prod_{\sigma\in\Sigma}\kappa_\sigma(2\pi i)\psi_{Z}(-k) G_{k, \nu, \zeta_\kappa, F}\left(z, \frac{k}{2}\right),\nonumber
\end{align}
and for any $d = \left(d_v\right)_{v\in\Sigma}\in\ZZ_{\geq 0}^\Sigma$,
\begin{align*}
\left(\kappa_{k+2d, \nu-d}\right)^{-1}&(c)\int_{\mathcal{G}_n}\frac{1}{\mathbf{N}_{k, \nu}\left(x\mathbf{N}_{K/E}(x)^{-n}\det y\right)}F\left(x, y^{-1}\right){\det}\left(\mathbf{N}_{K/E}\left(x\right)y^{-1}\right)^d d\mu_{\mathfrak{b}, n}\left(\uA\right)\\
&=\left(\kappa_{k+2d, \nu-d}\right)^{-1}(\Omega)\prod_{\sigma\in\Sigma}(2\pi i)^{nd}\psi_{Z}(-k)G_{k+2d, \nu-d, F(x, y)}\left(z, \frac{k}{2}\right),
\end{align*}
where $z$ is a point in $\prod_{\sigma\in\Sigma}\hn$ corresponding to the CM abelian variety $\uA$ viewed as an abelian variety over $\IC$ (by extending scalars to $\IC$) and $\Omega$ and $c$ are the periods from Equation \eqref{AlgCMequ}.  Here, $Z$ is the irreducible subrepresentation of $\prod_{\sigma\in\Sigma}\gln(\IC)\times\gln(\IC)$ of highest weight $\kappa$ and has $\zeta_\kappa$ as a highest weight vector; by $\kappa(a)$ with $a$ a scalar, we mean $\kappa$ evaluated at the $n$-tuple $(a, \ldots, a)$ in the torus.  

In other words, the $p$-adic measure $\mu_{\mathfrak{b}, n}$ allows us to $p$-adically interpolate the values of the $\ci$- (not necessarily holomorphic) function $G_{k, \nu, \zeta_\kappa, F}\left(z, \frac{k}{2}\right)$ at CM points $z$.
\begin{thm}\label{AVmeasure}
For each ordinary abelian variety $\uA$ defined over a (profinite) $p$-adic $\OK$-algebra $R_0$, there is an $R_0$-valued $p$-adic measure $\mu(\uA):=\mu_{\mathfrak{b}, n}(\uA)$ defined by
\begin{align*}
\int_{\mathcal{G}_n}H d\mu_{\mathfrak{b}, n}(\uA) = G_{n, 0, F}(\uA)
\end{align*}
for all continuous $R$-valued functions $H$ on $\mathcal{G}_n$, with
\begin{align*}
F\left(x, y\right) & = \frac{1}{\mathbf{N}_{n, 0}\left(x^{-1}\mathbf{N}_{K/E}(x)^{n}\det y\right)}H\left(x, y^{-1}\right)
\end{align*}
extended by $0$ to all of $\left(\OK\otimes\ZZ_p\right)\times M_{n\times n}\left(\Oe\otimes\ZZ_p\right)$.
When $R_0 =\varprojlim_mR/p^mR$ with $R\subseteq\bar{\IQ}$, $\uA$ is an ordinary CM point defined over $R$, and $F$ is a locally constant function supported on $\left(\OK\otimes\ZZ_p\right)^\times\times \gln\left(\Oe\otimes\ZZ_p\right)$ satisfying
\begin{align*}
F\left(ex, \mathbf{N}_{K/E}(e)^{-1}y\right) = \mathbf{N}_{k, \nu}(e)F\left(x, y\right)
\end{align*}
for all $e\in \OK^\times$, $x\in \OK\otimes\ZZ_p$, and $y\in \gln\left(\Oe\otimes\ZZ_p\right)$,
\begin{align*}
\left(\kappa\cdot\kappa_{k, \nu}\right)^{-1}&(c)\int_{\mathcal{G}_n}\frac{1}{\mathbf{N}_{k, \nu}\left(x\mathbf{N}_{K/E}(x)^{-n}\det y\right)}F\left(x, y^{-1}\right)\zeta_\kappa\left(\mathbf{N}_{K/E}\left(x\right)y^{-1}\right) d\mu_{\mathfrak{b}, n}(\uA)\\
& = \left(\kappa\cdot\kappa_{k, \nu}\right)^{-1}(\Omega)\prod_{\sigma\in\Sigma}\kappa_\sigma(2\pi i)\psi_{Z}(-k) G_{k, \nu, \zeta_\kappa, F}\left(z, \frac{k}{2}\right)
\end{align*}
with $z\in\prod_{v\in\Sigma}\hn$ corresponding to the ordinary CM abelian variety $\uA$ viewed as an abelian variety over $\IC$.
\end{thm}

The pullback of an automorphic form on $U(n,n)$ to $U(n)\times U(n)$ is automatically an automorphic form on the product of definite unitary groups $U(n)\times U(n)$.  So Theorem \ref{emeasurethm} also gives a $p$-adic measure with values in the space of automorphic forms on the product of definite unitary groups $U(n)\times U(n)$.  In \cite[Section 4]{emeasurenondefinite}, we explain how to modify our construction to obtain $p$-adic measures with values in the space of automorphic forms on certain non-definite groups.

\begin{rmk}[Relationship to the Eisenstein Measures in Section 4 of \cite{AppToSHL}]\label{rmkapptoshl}
For the curious reader, although we shall not need this remark anywhere else in this paper, we briefly explain the relationship between the measure $\mu_{\mathfrak{b}, n}$ defined in Theorem \ref{emeasurethm} and the measure $\phi$ defined in \cite[Theorem 20]{AppToSHL}.  For each $v\in\Sigma$, let $r_v = r(v)$ be a positive integer $\leq n$, and let $r = \left(r_v\right)_v\in\ZZ^\Sigma$.  As in \cite[Equation (33)]{AppToSHL}, let 
\begin{align}\label{Tofr}
T(r) = \prod_{v\in\Sigma}{\underbrace{{\Oe}_v^\times\times\cdots \times{\Oe}_v^\times}_{r_v \mbox{ copies}}}.
\end{align}
Let $\rho = \prod_{v\in\Sigma}\left(\rho_{1, v}, \ldots, \rho_{r(v), v}\right)$ be a $p$-adic character on $T(r)$ (i.e. $\rho\left(\left(\alpha_v\right)_{v\in\Sigma}\right) := \prod_{v\in\Sigma}\prod_{i = 1}^{r(v)}\rho_{i,v}\left(\alpha_v\right)$ for all $\alpha = \left(\alpha_v\right)_{v\in\Sigma}\in T(r)$), let $n=n_{1, v}+\cdots+ n_{r_v, v}$ be a partition of $n$ for each $v\in\Sigma$, and let $F_{\rho}$ be the function on $M_{n\times n}(E)$ defined by 
\begin{align*}
F_{\rho}(x) := \prod_{v\in\Sigma}\prod_{i=1}^{r(v)}\rho_{i, v}\left({\det}_{n_i}(x)\right),
\end{align*}
with $\det_j$ defined as on page \pageref{detidefn}.  Let $\chi$ be a $p$-adic function supported on $\left(\OK\otimes\ZZ_p\right)^\times/\overline{\OK^\times}$ and extended by $0$ to all of $\OK\otimes\ZZ_p$ .  Let $H_{\rho, \chi}$ be the function corresponding via the bijection in Lemma \ref{HFLem} to the function $F_{\rho, \chi}$ supported on $\mathcal{G}_n$ (and extended by $0$) defined by
\begin{align*}
F_{\rho, \chi}(x, y) = \chi(x)\mathbb{N}_{n, 0}(x)F_{\rho}(\mathbf{N}_{K/E}(x){ }^t y).
\end{align*}
Then
\begin{align*}
\int_{\mathcal{G}_n}H_{\rho, \chi}d\mu_{\mathfrak{b},n} = \int_{\left(\OK\otimes\ZZ_p\right)^\times/\overline{\OK}^\times\times T(r)} (\chi, \rho)d\phi.
\end{align*}
\end{rmk}

Note that the measure $\phi$ is dependent upon the choice of $r$ and the choice of the partition of $n$, while the measure $\mu_{\mathfrak{b}, n}$ is independent of both of these choices.

\section{Remarks about the case of symplectic groups, Siegel modular forms, and Katz's Eisenstein Measure for Hilbert Modular forms}\label{symplectickatz}

Note that the case of Siegel modular forms is quite similar.  We essentially just need to replace the CM field $K$ with the totally real field $E$ throughout.  Once we have replaced $K$ by $E$, $\mathbf{N}_{k, \nu}$ becomes $\mathbf{N}_{E/\IQ}^k$, and $\mathbf{N}_{K/E}$ becomes the identity map.  Consequently, Equations \eqref{HF1} and \eqref{HF2} become
\begin{align*}
H\left(x, y\right) & = \frac{1}{\mathbf{N}_{E/\IQ}\left(x^{1-n}\det y\right)^n}F\left(x, y^{-1}\right)\\
F\left(x, y\right) & = \frac{1}{\mathbf{N}_{E/\IQ}\left(x^{-1+n}\det y\right)^n}H\left(x, y^{-1}\right).
\end{align*}
To highlight the similarity with \cite[Section 4.2]{kaCM}, we note that when $n=1$, these equations become
\begin{align*}
H\left(x, y\right) & = \frac{1}{\mathbf{N}_{E/\IQ}\left(y\right)}F\left(x, y^{-1}\right)\\
F\left(x, y\right) & = \frac{1}{\mathbf{N}_{E/\IQ}\left( y\right)}H\left(x, y^{-1}\right).
\end{align*}
This relationship between $H$ and $F$ is similar to the relationship between the functions denoted $H$ and $F$ by Katz in \cite[Section 4.2]{kaCM}, which play a similar role to the functions we denoted by $H$ and $F$.  (The minor difference between Katz's relationship between $H$ and $F$ and ours is due the fact that throughout the paper, his $F(x, y)$ is our $F(y, x)$, i.e. our first variable plays the role of his second variable and vice versa throughout the paper.)

The differential operators are developed from the $\ci$-perspective simultaneously for both unitary and symplectic groups in \cite[Section 12]{shar}.  As noted on \cite[p. 4]{EDiffOps}, in \cite[Section 3.1.1]{EDiffOps}, and in \cite{padiffops1, padiffops2}, the algebraic geometric and $p$-adic formulation of the operators for Siegel modular forms (i.e. for symplectic groups) is similar.  In the case of Siegel modular forms, the algebraic geometric formulation of the differential operators is discussed in \cite[Section 4]{hasv}.  Also, the case of symplectic groups is handled directly alongside the case of unitary groups in Hida's discussion of $p$-adic automorphic forms in \cite[Chapter 8]{hida}.  So the construction in this paper carries over with only minor changes (essentially, replacing $K$ by $E$ throughout) to the case of symplectic groups over a totally real field $E$ and automorphic forms (Siegel modular forms) on those groups.

\subsection{The case $n=1$}\label{kacmnequals1}
Continuing with the symplectic case with $n=1$, Theorem \ref{proplc} becomes
\begin{thm}\label{proplckatz}
Let $R$ be an $\Oe$-algebra, let and let $k\geq 1$ be an integer.    For each locally constant function 
\begin{align*}
F: \left(\Oe\otimes\ZZ_p\right)\times \left(\Oe\otimes\ZZ_p\right)\rightarrow R
\end{align*}
supported on $\left(\Oe\otimes\ZZ_p\right)^\times\times \left(\Oe\otimes\ZZ_p\right)^\times$ which satisfies
\begin{align}\label{equnknualakacm2}
F\left(ex, e^{-1}y\right) = \mathbf{N}_{E/\IQ}(e)^kF\left(x, y\right)
\end{align}
for all $e\in \Oe^\times$, $x\in \Oe\otimes\ZZ_p$, and $y\in \Oe \otimes\ZZ_p$, there is a Hilbert modular form $G_{k, F}$ of weight $k$ defined over $R$ whose $q$-expansion at a cusp $m\in GM$ is of the form $\sum_{\beta>0}c(\beta)q^\beta$ (where $L_{m}$ is the lattice in $E$ determined by $m$), with $c(\beta)$ a finite $\ZZ$-linear combination of terms of the form
\begin{align*}
F\left(a, (a)^{-1}\beta\right)\mathbf{N}\left(a^{-1}\beta\right)^k\mathbf{N}_{E/\IQ}\left(\beta\right)^{-1}
\end{align*}
(where the linear combination is a sum over a finite set of $p$-integral $a\in E$ dependent upon $\beta$ and the choice of cusp $m\in GM$). 
\end{thm}
Still continuing with the symplectic case with $n=1$, Theorem \ref{emeasurethm} becomes
\begin{thm}\label{emeasurethmkatz}
There is a measure $\mu$ on 
\begin{align*}
\mathcal{G}=\left(\left(\Oe\otimes\ZZ_p\right)^\times\times \left(\Oe\otimes\ZZ_p\right)^\times\right)/\overline{\Oe^\times}
\end{align*}
 (with values in the space of $p$-adic Hilbert modular forms) defined by
\begin{align*}
\int_{\mathcal{G}}H d\mu = G_{1, F}
\end{align*}
for all continuous $R$-valued functions $H$ on $\mathcal{G}$, with
\begin{align*}
F\left(x, y\right) & = \frac{1}{\mathbf{N}_{E/\IQ} (y)}H\left(x, y^{-1}\right)
\end{align*}
extended by $0$ to all of $\left(\Oe\otimes\ZZ_p\right)\times\left(\Oe\otimes\ZZ_p\right)$. 
\end{thm}
Note that we have essentially recovered the Eisenstein series and measure from \cite[Definition (4.2.5)]{kaCM}.  (Again, the difference between Katz's order of the variables $x$ and $y$ and ours is due to the fact that throughout the paper, his $F(x, y)$ is our $F(y, x)$, i.e. our first variable plays the role of his second variable and vice versa throughout the paper.)  The reader familiar with \cite{kaCM} will notice the similarities with \cite[(5.5.1)-(5.5.7)]{kaCM}.  In particular, let $\chi$ be a Gr\"ossencharacter of the CM field $K$ whose conductor divides $p^\infty$ and whose infinity type is
\begin{align*}
-k\sum_{\sigma\in\Sigma}\sigma -\sum_{\sigma\in\Sigma}d(\sigma)\left(\sigma-\overline{\sigma}\right)
\end{align*}
with $d(\sigma)\geq 0$ for all $\sigma\in\Sigma$ and $k\geq n$.
We view $\chi$ as an $\OCp$-valued character on $\adeles^{\infty,\times}\times\prod_{v\in\Sigma}\overline{\IQ}$ (by restricting it to this group) and consider its restriction to the subring consisting of elements $((1_v)_{v\ndivides p\infty}, a, a)$, with $a\in \OK\otimes\ZZ_{(p)}$, which is a subring of 
\begin{align*}
\left(\OK\otimes\ZZ_p\right)^\times\isomto \left(\Oe\otimes\ZZ_p\right)^\times\times\left(\Oe\otimes\ZZ_p\right)^\times.
\end{align*}
Then we have
\begin{align*}
\chi\left(\alpha\right) &= \chi_{\mathrm{finite}}\left(\alpha\right)\cdot\frac{\prod_{\sigma\in\Sigma}\sigma(\bar{\alpha})^{d(\sigma)}}{\prod_{\sigma\in\Sigma}\sigma(\alpha)^{k+d(\sigma)}}\\
\chi\left(x, y\right) & =  \chi_{\mathrm{finite}}\left(x, y\right)\cdot\frac{\prod_{\sigma\in\Sigma}\sigma(x)^{d(\sigma)}}{\prod_{\sigma\in\Sigma}\sigma(y)^{k+d(\sigma)}},
\end{align*}
with $\chi_{\mathrm{finite}}$ a locally constant function.
If
\begin{align*}
F(x, y) &= \frac{1}{\mathbf{N}(y)}\chi\left(x, \frac{1}{y}\right)\label{ksim1},\\
& = \chi_{\mathrm{finite}}\left(x, \frac{1}{y}\right)\cdot\mathbf{N}(y)^{k-1}\prod_{\sigma\in\Sigma}\sigma\left(xy\right)^{d(\sigma)},
\end{align*}
then
\begin{align}
\int_{\mathcal{G}} \chi\left(x, y\right)d\mu_{\mathfrak{b}, 1} &= G_{1, F}\\
& = G_{1, \chi_{\mathrm{finite}}\left(x, \frac{1}{y}\right)\mathbf{N}(y)^{k-1}\prod_{\sigma\in\Sigma}\sigma\left(xy\right)^{d(\sigma)}}\\
& = G_{k, \chi_{\mathrm{finite}}\left(x, \frac{1}{y}\right)\prod_{\sigma\in\Sigma}\sigma\left(xy\right)^{d(\sigma)}}\\
&= \left(\prod_{\sigma\in\Sigma}\theta\left(\sigma\right)^{d(\sigma)}\right)\left(G_{k, \chi_{\mathrm{finite}\left(x, \frac{1}{y}\right)}}\right),\label{ksim2}
\end{align}
where $\theta(\sigma)$ denotes the ($\sigma$-component of the) differential operator $\theta(\det)$ acting on automorphic forms in the $1$-dimensional, symplectic case.
Note the similarity of Equations \eqref{ksim1} through \eqref{ksim2} with \cite[Equations (5.5.6)-(5.5.7)]{kaCM}.

\bibliography{eischen}   

\end{document}